\documentclass[12pt]{amsart}
\usepackage[utf8]{inputenc}
\usepackage[T1]{fontenc}
\usepackage{upref, amssymb}
\usepackage[all]{xy}
\usepackage{hyperref}
\usepackage[a4paper]{geometry}

\newcommand{\Z}{{\mathbb Z}}
\newcommand{\R}{{\mathbb R}}
\newcommand{\Q}{{\mathbb Q}}
\newcommand{\T}{{\mathbb T}}
\newcommand{\N}{{\mathbb N}}
\newcommand{\G}{{\mathbb G}}

\newcommand{\CC}{{\mathcal C}}
\newcommand{\CI}{{\mathcal I}}

\newcommand{\bu}{{\mathbf u}}
\newcommand{\bx}{{\mathbf x}}
\newcommand{\by}{{\mathbf y}}

\newcommand{\inv}{^{-1}}
\newcommand{\zrat}{Z_{\text{\rm rat}}}
\newcommand{\rrat}{R_{\text{\rm rat}}}
\newcommand{\zab}{Z_{\text{\rm ab}}}
\newcommand{\rab}{R_{\text{\rm ab}}}
\newcommand{\riso}{R_{\text{\rm iso}}}
\newcommand{\ziso}{Z_{\text{\rm iso}}}
\newcommand{\risoc}{R_{\text{\rm isoc}}}
\newcommand{\zisoc}{Z_{\text{\rm isoc}}}

\DeclareMathOperator{\range}{range}
\DeclareMathOperator{\torsion}{Torsion}
\DeclareMathOperator{\diameter}{diameter}
\DeclareMathOperator{\id}{id}
\DeclareMathOperator{\tor}{Tor}
\DeclareMathOperator{\ext}{Ext}
\renewcommand{\hom}{\mathop{\rm Hom}\nolimits}
\DeclareMathOperator{\SL}{SL}
\DeclareMathOperator{\gab}{G_\text{ab}}

\newtheorem{theorem}{Theorem}[section]
\newtheorem{lemma}[theorem]{Lemma}
\newtheorem{proposition}[theorem]{Proposition}
\newtheorem{corollary}[theorem]{Corollary}
\newtheorem{fact}[theorem]{Fact}
\theoremstyle{remark}
\newtheorem{remark}[theorem]{Remark}
\newtheorem*{comments}{Comments}
\theoremstyle{definition}
\newtheorem{definition}[theorem]{Definition}

\begin{document}

\title[Extensions and dimension groups]{Extensions of Cantor minimal systems and  dimension groups}


\author{Eli Glasner and Bernard Host} 

\address{Department of Mathematics,
Tel-Aviv University, Ramat Aviv, Israel}
\email{glasner@math.tau.ac.il}

\address{\'Equipe d'analyse et de math\'{e}matiques appliqu\'{e}es,
Universit\'{e} de Paris-Est Marne-la-Vall\'{e}e, 77454 Marne-la-Vall\'ee Cedex 2, France}
\email{bernard.host@univ-mlv.fr}

\date{\today}


\thanks{The first author thanks the Israel Science Foundation (grant
number 4699). The second author was partially founded by the Institut universitaire de France.}

\thanks{{\em 2000 Mathematical Subject Classification  54H20, 
37B05, 19D55}}

\begin{abstract}
Given a factor map $p\colon(X,T)\to(Y,S)$ of Cantor minimal systems, we study the relations between the dimension groups of the two systems.

First, we interpret the torsion subgroup of the quotient of the dimension groups
$K_0(X)/K_0(Y)$ in terms of intermediate extensions which are extensions of 
$(Y,S)$ by a compact abelian group.
Then we show that, by contrast, the existence of an intermediate non-abelian 
finite group extension can produce a situation where the dimension group of 
$(Y,S)$ embeds into a proper  subgroup of the dimension group of $(X,T)$,
yet the quotient of the dimension groups is nonetheless torsion free.

Next we define higher order cohomology groups $H^n(X \mid Y)$
associated to an extension, and study them in various cases (proximal extensions, extensions by, 
not necessarily abelian, finite groups, etc.).
Our main result here is that all the cohomology groups $H^n(X \mid Y)$
are torsion groups. As a consequence we can now identify $H^0(X \mid Y)$
as the torsion group of $ K_0(X)/K_0(Y)$.
\end{abstract}

\maketitle

\tableofcontents

\section*{Introduction}
In this work a {\em dynamical system} is a pair $(X,T)$ where $X$ is a compact metric space and $T \colon X \to X$ is a  homeomorphism. 
If $X$ is a Cantor space we say that $(X,T)$ is a \emph{Cantor system}. If the orbit $\{T^nx\colon n\in\Z\}$ of every point $x\in X$ is dense, we say that $(X,T)$ is \emph{minimal}.
 For the definition of the dimension group which is associated to
a minimal Cantor system and for more details on this subject we refer the reader to the papers \cite{HPS}, \cite{GPS},
\cite{GW},  \cite{O} and \cite{DHS}.
For further developments see \cite{GPS1} and \cite{S}.

A \emph{factor map}   
$\pi\colon (X,T) \to (Y,S)$ 
is a surjective continuous map $\pi\colon X \to Y$ such that $\pi\circ T = S\circ\pi$. 
In this case we say that $(Y,S)$ is a factor of $(X,T)$. Depending on the point of view, we say also that $(X,T)$ is an extension of $(Y,S)$ and then the map $\pi$ is called an \emph{extension}.

Given a factor map $p\colon(X,T)\to(Y,S)$ of Cantor minimal systems, our
goal is to study the relations between the dimension groups of the two systems. 
Following ideas and methods developed by 
Herman, Putnam and Skau \cite{HPS}, Giordano Putnam and Skau \cite{GPS},
and Glasner and Weiss \cite{GW}, we show in Sections \ref{sec:abelian}, 
\ref{proof} (and Appendix~\ref{sec:appendix})
that an extension of minimal dynamical systems $p\colon (X,T) \to (Y,S)$ yields a nontrivial torsion in the quotient group $K_0(X,T)  / $ $p^*K_0(Y,S)$ of the corresponding dimension groups if and only if  there is a nontrivial intermediate cyclic group extension $X \to Z \to Y$.
More generally, we interpret the torsion subgroup of the quotient of the dimension groups in terms of intermediate extensions which are extensions of $(Y,S)$ by a totally disconnected compact abelian group.
In section \ref{NA}  and Appendix~\ref{app:B} we show that, by contrast, the existence of an intermediate non-abelian finite group extension can produce a
situation where $p^*K_0(Y,S)\subsetneq K_0(X,T)$
yet $K_0(X,T) / p^*K_0(Y,S)$ is nevertheless torsion free.

Next (in Section \ref{sec:cohomology}) we define higher order cohomology groups $H^n(X \mid Y)$
associated to an extension, and study them in various cases (proximal extensions, extensions by, not necessarily abelian, finite groups, etc.).
Our main result here (Section \ref{sec:torsion}) is that all the cohomology groups $H^n(X \mid Y)$
are torsion groups. As a consequence we can now identify $H^0(X \mid Y)$
as the torsion group $\torsion \bigl( K_0(X)/K_0(Y) \bigr)$. 
As an example, in Appendix~\ref{ap:morse} we compute some of these groups in the case of the Morse dynamical system.

Most of the homological algebra we utilize is surely well known and classical. However, because we are using it in the special context
of the spaces $\CC(X)=\CC(X,\Z)$ where $X$ is a Cantor set,
finding an exact reference in the literature is not an easy task. We instead
provide mostly complete proofs. This practice serves two 
purposes. First, it saves us the need to search for exact sources,
and second, hopefully, it will enable those readers who lack the necessary algebraic background to read the paper with relative ease.

We thank Christian Skau for his stimulating questions. 
We also thank Benjy Weiss and Avinoam Mann for their contributions to this work.

\section{Notation and preliminaries}

\subsection{The dimension group $K_0(X,T)$}
If $X$ is a Cantor space, we write $\CC(X)$ for the additive group of continuous functions on $X$ with integer values. If $X$ and $Y$ are Cantor spaces and $p\colon X\to Y$ a continuous map,
 we write $p^*\colon \CC(Y)\to\CC(X)$ the map $f\mapsto f\circ p$. This map is a group homomorphism and is 
one-to-one if $p$ is onto.
 
A \emph{Cantor dynamical system} $(X,T)$ is a Cantor set $X$ endowed with a homeomorphism
$T\colon X\to X$. 

As above a \emph{factor map} (or an {\em extension}) from a  Cantor dynamical system $(X,T)$ to a Cantor dynamical system $(Y,S)$  is a map $p\colon X\to Y$, continuous and onto, 
with $S\circ p=p\circ T$.

Let $(X,T)$ be Cantor dynamical system.
The \emph{coboundary map} $\beta\colon
\CC(X)\to\CC(X)$ is defined by 
$\beta=T^*-\id$, that is
$$
\beta f=f\circ T-f\text{ for every }f\in\CC(X)\ .
$$
A function $f\in\CC(X)$ is a \emph{coboundary} if it belongs the range $\beta\CC(X)$ of the coboundary map.
We define
$K_0(X,T)$ to be the quotient group
$$
K_0(X,T):= \frac {\CC(X)}{\beta\CC(X)}\ .
$$
For $f\in\CC(X)$ we write $[f]$ for its image in $K_0(X,T)$.

The group $K_0(X,T)$ can be given the order induced by the natural order of $\CC(X)$ but we do not use this order here.

The next Lemma is classical~\cite{HPS}.
\begin{lemma}
\label{lem:K0torsionfree}
Let $(X,T)$ be Cantor dynamical system. Then $K_0(X,T)$ is torsion-free.
\end{lemma}

\begin{proof} 
Let $\alpha\in K_0(X,T)$ and let $k>0$ be an integer with $k\alpha=0$. We choose $f\in\CC(X)$ with $[f]=\alpha$. Then $[kf]=0$ and there exists $g\in\CC(X,Z)$ with $kf=g\circ T-g$. We write
$$
g=kh+q\text{ with }h,q\in\CC(X)\text{ and }0\leq q<k\ .
$$
We get $q\circ T-q=k(f-h\circ T+h)$ and $-k< q\circ T-q<k$ and thus $q\circ T-q=0$.
We conclude that $f=h\circ T-h$.
\end{proof}

\subsection{Dimension groups and factors}

Let $p\colon(X,T)\to(Y,S)$ be a factor map between Cantor dynamical systems. The homomorphism $p^*\colon \CC(Y)\to\CC(X)$ satisfies 
$\beta\circ p^*=p^*\circ\beta$ and thus induces a group homomorphism, written $p^*$ also, from $K_0(Y,S)$ to $K_0(X,T)$. 
Proposition 3.1 of \cite{GW} shows that this is a natural embedding of 
$K_0(Y,S)$ into $K_0(X,T)$ which is order and order unit preserving.
For completeness we provide a proof that $p^*$ is  
one-to-one.

\begin{lemma} 
\label{lem:onetoone}
Let  $p\colon(X,T)\to(Y,S)$ be a factor map between Cantor minimal dynamical systems.
Then  $p^*\colon K_0(Y,S) \to K_0(X,T)$ is  
one-to-one.
\end{lemma}

\begin{proof}
Let $\alpha\in K_0(Y)$ be such that $p^*\alpha=0$ and let $f\in\CC(Y)$ be such that $[f]=\alpha$. Then $[f\circ p]=p^*\alpha=0$, that is, $f\circ p$ is a coboundary of the system $(X,T)$ and there exists $g\in\CC(X)$ with $f\circ p=g\circ T-g$. For every $n>0$ and every $y\in Y$, choosing $x\in X$ with $p(x)=y$ we have
$$
f^{(n)}(y):=\sum_{j=0}^{n-1}f(S^jy)=\sum_{j=0}^{n-1}f\circ p(T^jx)=g(T^nx)-g(x)
$$
and thus the family of functions $\bigl(f^{(n)}\colon n>0\bigr)$ is uniformly bounded by $2\Vert g\Vert_\infty$. By~\cite{GH}, $f$ is a coboundary of the system $(Y,S)$. We conclude that $\alpha=0$.
\end{proof}

\section{Abelian intermediate extensions}
\label{sec:abelian}

Here $p\colon(X,T)\to(Y,S)$ is a factor map between minimal Cantor dynamical systems.

\subsection{Intermediate extensions}
An \emph{intermediate extension}   $(X,T)\overset{r}{\rightarrow} (Z,R)\overset{q}{\rightarrow}(Y,S)$ consists of a Cantor system 
$(Z,R)$ and two factor maps $q\colon(Z,R)\to(Y,S)$ and $r\colon(X,T)\to(Z,R)$ with  $q\circ r=p$.

Let  $(X,T)\overset{r}{\rightarrow} (Z,R)\overset{q}{\rightarrow}(Y,S)$ and 
$(X,T)\overset{r'}{\rightarrow} (Z',R')\overset{q'}{\rightarrow}(Y,S)$ be two intermediate extensions. We say that the first one is \emph{over} (or \emph{greater than}) the second one if there exists a factor map $\phi\colon (Z,R)\to(Z',R')$ with $\phi\circ r=r'$. 
This implies that $q'\circ\phi=q$ and we have the commutative   diagram:
$$
\xymatrix{
(X,T)\ar[rr]^p \ar[rd]^r\ar[rdd]_{r'} & & (Y,S) \\
& (Z,R) \ar[ru]^q \ar[d]^\phi & \\
& (Z',R') \ar[ruu]_{q'} & 
}
$$

We remark that the factor map $\phi$ with these properties is then unique. We say that the two intermediate extensions are \emph{isomorphic} if the factor map $\phi$ is an isomorphism. We identify isomorphic intermediate extensions. The relation ``being over'' defined above is then a
partial order on the family of intermediate extensions.
 
Every family of intermediate extensions admits a \emph{supremum} for this order; that is, an intermediate extension which is larger than every intermediate extension of the family, and smaller than every other intermediate extension with this property. We give the construction  in the case of a family of two intermediate extensions, the general is is similar.
 
Let  $(X,T)\overset{r}{\rightarrow} (Z,R)\overset{q}{\rightarrow}(Y,S)$ and 
$(X,T)\overset{r'}{\rightarrow} (Z',R')\overset{q'}{\rightarrow}(Y,S)$ be two intermediate extensions. Let $Z''=\{ (r(x),r'(x))\colon x\in X\}\subset Z\times Z'$. Then $Z''$ is a closed subset of $Z\times Z'$ and thus is compact. The map $R''=R\times R'$ is a homeomorphism of $Z''$ onto itself and the map $r''=(r,r')\colon X\to Z''$ is a factor map from $(X,T)$ to $(Z'',R'')$. We define a map $q''\colon Z''\to Y$ by $q''(z,z')=q(z)=q'(z')$. This map is clearly a factor map. We have constructed an intermediate extension
$(X,T)\overset{r''}{\rightarrow} (Z'',R'')\overset{q''}{\rightarrow}(Y,S)$ satisfying all the announced properties.
Sometimes we denote this extension by $Z \vee Z'$.

\subsection{Intermediate extensions by compact abelian groups}
\label{subsec:inter_ab}
We say that 
$(X,T)\to (Z,R)\overset{q}{\rightarrow}(Y,S)$ is an  \emph{intermediate extension by a compact group $G$} if 
$G$ is a compact subgroup of the Polish group ${\rm {Homeo}}(X)$
which commutes with $T$ and the extension $q$ is isomorphic to the quotient map $X \to X/G$.

We consider here the case that $G$ is a finite or Cantor  abelian group. 

If an intermediate extension is smaller than an intermediate extension by a compact abelian group 
$G$,   then it is itself an intermediate extension by a compact abelian group $H$, which is a quotient of the group $G$. In particular, if $G$ is finite or Cantor, then $H$ is finite or Cantor too.
It is easy to see that the supremum of any family of extensions by compact abelian groups is also an extension by a compact abelian group. 
If all the groups of the family are finite or Cantor, then the supremum is finite or Cantor too.

\subsection{Intermediate extensions by finite and Cantor groups}

\begin{theorem}
\label{th:finite-abel}
Let $p\colon(X,T)\to(Y,S)$ be a factor map between Cantor dynamical systems. Then there exists a bijection between the family of finite subgroups of $K_0(X,T)/p^*K_0(Y,S)$ and the family of intermediate extensions $(X,T)\overset{r}{\rightarrow}(Z,R)\overset{q}{\rightarrow}(Y,S)$ where 
$q\colon (Z,R)\to(Y,S)$ is an extension by a finite abelian group.

This bijection maps each finite subgroup $G$ of  $K_0(X,T)/p^*K_0(Y,S)$ to an extension of $(Y,S)$ by the dual group $\widehat G$ of $G$.

Moreover, this bijection preserves the orders: If $G$ and $G'$ are finite subgroups  of 
$K_0(X,T)/p^*K_0(Y,S)$ with $G'\subset G$, then the intermediate extension associated to $G'$ is below the intermediate extension associated to $G$.
\end{theorem}

In the special case when the finite group is 
$\Z_n:=\Z/n\Z$, 
Theorem  \ref{th:finite-abel}
says that $p \colon (X,T) \to (Y,S)$ admits no nontrivial intermediate 
$\Z_n$-extension if and only if the group  $K_0(X,T)/p^*K_0(Y,S)$ is 
torsion free. We give an easy direct proof of this fact in Appendix~\ref{sec:appendix}.

Notice that the condition that no intermediate nontrivial finite cyclic group
extension of $Y$ exists, is clearly
satisfied when $p$ is either a weakly mixing extension
or when it is a proximal extension (see Definition~\ref{def:extension} below).
 Hence, for such an extension,  if
$K_0(X,T)/\allowbreak p^*K_0(Y,S)$ is nontrivial 
it is torsion free.
This latter assertion is a result of Giordano Putnam and Skau 
(see~\cite{GPS1}).
For later use we state this fact as a formal corollary:
 
\begin{corollary}\label{cor:wm}
For an extension $p : (X,T) \to (Y,S)$ between Cantor minimal systems which
is either weakly mixing or proximal, we have
$$
\torsion (K_0(X,T)/p^*K_0(Y,S)) = 0.
$$
 \end{corollary}

\begin{theorem}
\label{th:abel-max}
Let $p\colon(X,T)\to(Y,S)$ be a factor map between Cantor dynamical systems. Let $H$ be the torsion subgroup of $K_0(X,T)/p^*K_0(Y,S)$ and $K$ the dual group of $H$; recall that $K$ is a Cantor abelian group.

Then there exists an intermediate extension $(X,T)\overset{r}{\rightarrow} (\zab,\rab)\overset{q}{\rightarrow}(Y,S)$ where 
$q\colon (\zab,\rab)\to(Y,S)$ is an extension by $K$.

Moreover, the family of intermediate extensions which are extensions of $(Y,S)$ by a compact 
finite or Cantor abelian group coincides with the family of intermediate extensions below $(\zab,\rab)$.
\end{theorem}

\begin{proof}[Proof of Theorem~\ref{th:abel-max} assuming Theorem~\ref{th:finite-abel}]

Let $(X,T)\overset{r}{\rightarrow} (\zab,\rab)\overset{q}{\rightarrow}(Y,S)$ be the supremum of all intermediate extensions by finite or Cantor  abelian groups. By preceding remarks, $(\zab,\rab)$ is an extension of $(Y,S)$ by a compact finite or Cantor abelian group $K$, and this extension satisfies the second statement of the theorem.

Since $K$ is either finite or Cantor, it can be written as an inverse limit of a sequence 
$K_1\leftarrow K_2\leftarrow\cdots\leftarrow K$ of finite 
abelian groups. It follows that $(\zab,\rab)$ is the supremum of an increasing sequence of intermediate extensions $(Y,S)\leftarrow (Z_1,R_1)\leftarrow (Z_2,R_2)\leftarrow\cdots\leftarrow(X,T)$, where $(Z_i,R_i)$ is an extension of $(Y,S)$ by $K_i$.

By Theorem~\ref{th:finite-abel}, $ K_0(X,T)/p^*K_0(Y,S)$ contains an increasing sequence $H_1\subset H_2\subset\cdots$ of finite abelian groups, with $\widehat{H_i}=K_i$ for every $i$. Let $H$ be the union of the groups $H_i$. Then we have $\widehat H=K$.

Since $H$ is a union of finite groups it is included in the torsion subgroup of 
$ K_0(X,T)/p^*K_0(Y,S)$ and it remains only to show the converse inclusion that is, to prove that every finite subgroup $G$ of $ K_0(X,T)/p^*K_0(Y,S)$ is included in $H$.
Let $(Z_G,R_G)$ be the intermediate extension associated to $G$ by Theorem~\ref{th:finite-abel}. It is an extension of $(Y,S)$ by the group $\widehat G$. This extension is below 
$(\zab,\rab)$ by definition of this last extension, 
and thus $\widehat G$ is a quotient of $K$. 
Therefore, there exists an $i$ such that $\widehat G$ is a quotient of $K_i$ (this can be seen by duality), and thus $(Z_G,R_G)$ is below $(Z_i,R_i)$. By the last statement of Theorem~\ref{th:finite-abel}, $G\subset H_i\subset H$ and we are done.
\end{proof}

\subsection{The absolute case}
\label{subsec:abs1}
By considering the case that $Y$ is  the trivial system consisting in a single point, we recover known results. 
The intermediate extensions between $Y$ and $X$ which are extensions of $Y$ by finite or Cantor abelian groups are simply the factors of $X$ that are rotations on finite or Cantor abelian groups. The maximal element $(\zrat,\rrat)$ in this family of factors is called the \emph{rational equicontinuous factor}. 
It is also the supremum of all factors of $(X,T)$ that are rotations on finite cyclic groups. It is also the factor of $(X ,T)$ spanned by the continuous eigenfunctions corresponding to rational eigenvalues, and the dual group $\widehat{\zrat}$ is the group of rational eigenvalues of $(X,T)$.

On the other hand, we have that $K_0(Y)=\Z$, with order unit $1$. The map $p^*\colon K_0(Y)\to K_0(X)$ is given by $k\mapsto ke_X$, where $e_X$ is the order unit of $K_0(X)$.
If we identify $\Z e_X$ and $\Z$, we have that the torsion subgroup of 
$K_0(X)/p^*K_0(Y)$ is the group of rational elements of $K_0(X)$ modulo $1$. By Theorem~\ref{th:abel-max} we get
$$
\widehat{\zrat}=\torsion\Bigl(\frac{(K_0(X)}{p^*K_0(Y)}\Bigr)
=\torsion\Bigl(\frac{K_0(X)}{\Z}\Bigr)
=\frac{ \{\alpha\in K_0(X)\colon \exists n>0,\ n\alpha\in\Z\} }{\Z}\ .
$$
 Since the dimension group is invariant under strong orbit equivalence
\cite{GPS},
we recover here the well known fact 
that the group of rational eigenvalues is invariant under strong orbit equivalence.
(See Ormes \cite[Theorem 2.2]{O}, where it is attributed to Giordano-Putnam-Skau, with a proof due to B. Host.)

We return to the absolute case in Section~\ref{subsec:abs2}.

\subsection{Non Abelian group extensions}\label{NA}

The natural generalization of the notion of an extension by a finite group is that of finite {\em isometric extension}. When $p : (X,T) \to (Y,S)$ is
an extension between Cantor minimal systems
this means that $X=Y\times G/H$ where $G$ is a finite group, $H$ a subgroup of $G$, and  $T(y,uH)=(Sy,\sigma(y)uH)$ for some $G$-cocycle $\sigma$. 
For a more detailed definition see Section~\ref{sec:isometric}.

Call an extension of dynamical systems
$U \overset \theta \to V$ {\it fi-free (finite isometric free)} if there is no
intermediate extension $U \to W\overset \eta \to V$ where $\eta$ is a
finite to one isometric extension. We say that $U \overset \theta   \to V$ is  
{\it fc-free (finite cyclic free)} if there is no intermediate extension
$U \to W\overset \eta \to V$ where $\eta$ is a finite cyclic group extension.

In this terminology Theorem \ref{th:finite-abel} states that {\em an extension 
$p : (X,T) \to (Y,S)$ between minimal Cantor
systems is fc-free if and only if $K_0(X,T)/p^*K_0(Y,S)$ is torsion free}.

\begin{theorem}\label{simple}
There exists an extension 
$p: (X,T) \to (Y,S)$ of Cantor minimal systems
which is a group extension with a finite
noncommutative simple group as the group of the 
extension (hence the extension is not fi-free),
for which the quotient group $K_0(X,T)/p^*K_0(Y,S)$ is (nontrivial and) 
torsion free.
\end{theorem}

\begin{proof}[Proof of Theorem \ref{simple}]
In the Appendix~\ref{app:B} below we show how one can construct
an extension $p\colon (X,T) \to (Y,S)$ of Cantor minimal systems
which is a group extension with a finite 
noncommutative simple group $G$ as the group of the 
extension and such that the affine map
$p_* : M_T(X) \to M_S(Y)$, induced on the 
simplexes of invariant probability measures, is not an isomorphism.
As these simplexes can be identified with the state spaces of
the corresponding dimension groups,
the latter property shows that the inclusion map 
$p^*: K_0(Y,S) \to K_0(X,T)$ is not surjective, whence the group  
$K_0(X,T)/p^*K_0(Y,S)$
is non-trivial.
This extension $p$, being a finite isometric extension, 
is clearly not fi-free and we next show
that nevertheless the quotient group
$K_0(X,T)/p^*K_0(Y,S)$ is torsion free.
Now if there exists an intermediate 
extension $X \to Z \to Y$ where 
$Z \to Y$ is a nontrivial cyclic group extension
with cyclic group $K=\Z_k$ as the group
of the extension, then $K$ is a homomorphic image
of $G$, hence trivial. Thus the extension $p$
is an fc-free extension and by Theorem \ref{th:finite-abel}, 
$K_0(X,T)/p^*K_0(Y,S)$, which is nontrivial, is torsion free. 
This completes the proof of the theorem.
\end{proof}

\section{A proof of Theorem~\ref{th:finite-abel}}\label{proof}

\subsection{$G$-cocycles}
Here,  $(Y,S)$ is a Cantor minimal system and $G$ is a finite abelian group, written with the additive notation. For more information on topological cocycles 
we refer to \cite{GH}, \cite{At} and \cite{LM}.
\begin{definition}
 A \emph{$G$-cocycle} of this system is a continuous map $\sigma\colon Y\to G$. The additive group of $G$-cocycles is written $\CC(Y,G)$.

For $\sigma\in\CC(Y,G)$ the \emph{coboundary} of $\sigma$ is the $G$-cocycle $\beta_G\sigma=\sigma\circ S-\sigma$. The range $\beta_G\CC(Y,G)$ is called the group of $G$-coboundaries.

Two $G$-cocycles are \emph{cohomologous} if their difference is a $G$-coboundary.

A $G$-cocycle $\sigma$ is \emph{ergodic} if it is not cohomologous to any cocycle with values in a proper subgroup $G'$ of $G$.

Two $G$-cocycles are \emph{equivalent} if there 
exists an automorphism $\phi$ of $G$ such that the cocycles $\phi\circ\sigma$ and $\tau$ are cohomologous.
\end{definition}

We remark that $\CC(Y,G)$ can be identified with $\CC(Y)\otimes G$ and we have $\beta_G=\beta\otimes\id_G$. Since $K_0(Y):=\CC(Y)/\beta\CC(Y)$ is torsion-free by Lemma~\ref{lem:K0torsionfree},
$(\beta\CC(Y))\otimes G$ can be considered as a subgroup of $\CC(Y)\otimes G$, and this subgroup is the range of $\beta\otimes\id_G$.
We get an isomorphism
$$
K_0(Y,S)\otimes G\cong\frac{\CC(Y,G)}{\beta_G\CC(Y,G)}\ ,
$$
and we identify these two groups.
Thus,  
a $G$-cocycle is ergodic if and only if its class in $K_0(Y,S)\otimes G$ does not belong to $K_0(Y,S)\otimes G'$ for any proper subgroup 
$G'$ of $G$.

Moreover, if 
$\alpha,\alpha'\in K_0(Y,S)\otimes G$ are the classes of two $G$-cocycles $\sigma,\sigma'\in\CC(Y,G)$, then these cocycles are equivalent if and only 
if there exists an automorphism $\phi$ of $G$ such that $\alpha'=\bigl(\id_{K_0(Y,S)}\otimes\phi\bigr)(\alpha)$.

\subsection{$G$-cocycles and extensions}
Here again,  $(Y,S)$ is a Cantor minimal system and $G$ is a finite abelian group, written with the additive notation. 

To every $G$-cocycle $\sigma\in\CC(Y,G)$ we associate an extension $q_G\colon(Y\times G,S_\sigma)\to(Y,S)$ by setting
$S_\sigma(y,g)=(Sy,g+\sigma(y))$ and $q_G(y,g)=y$.
We recall that this extension is minimal if and only if the cocycle $\sigma$ is ergodic. The proof of the next lemma is straightforward.
\begin{lemma}
\label{lem:equi}
Let $\sigma$ and $\tau$ be two ergodic $G$-cocycles on $(Y,S)$. Then these $G$-cocycles are equivalent if and only if the associated extensions are isomorphic (as extensions), meaning that there exists an isomorphism $\theta\colon (Y\times G,S_\sigma)\to(Y\times G,S_\tau)$ with $q_G\circ\theta=q_G$. In this case, the isomorphism $\theta$ has the form
$\theta(y,g)=(y, u(y)+\chi(g))$ where $\chi$ is an automorphism of $G$ and $u\in\CC(Y,G)$ satisfies $\tau(y)=u(Sy)-u(y)+\chi\circ\sigma(y)$. 
\end{lemma}

\subsection{$G$-cocycles and intermediate extensions}
\label{subsec:inter}
Now we assume that $p\colon (X,T)\to(Y,S)$ is a factor map between Cantor minimal systems.

Let $G$ be a finite abelian group and let $\sigma\in\CC(Y,G)$ be an ergodic 
$G$-cocycle. 
Assume first that the extension $p_G\colon(Y\times G,S_\sigma)\to(Y,S)$ is an intermediate extension below $(X,T)$, meaning that there exists a factor map 
$r\colon(X,T)\to(Y\times G,S_\sigma)$ with $q_G\circ r=p$. Then the map $r$ has the form $x\mapsto (p(x),\phi(x))$, where $\phi\colon X\to G$ satisfies $\phi\circ T-\phi=\sigma\circ p$. In other words, $\sigma\circ p$ is a $G$-coboundary of $(X,T)$.
We remark that any other factor map $r'\colon(X,T)\to(Y,S_\sigma)$ with $q_G\circ r'=p$ has the form $x\mapsto(p(x),\phi(x)+c)$ for some constant $c\in G$.

Conversely, let $\sigma\in\CC(Y,G)$ be an ergodic cocycle such that $\sigma\circ p$ is a 
$G$-coboundary. We chose $\phi\in\CC(X,G)$ with $\beta_G\phi=\sigma\circ p$ and define
$r\colon X\to(Y\times G)$ by $r(x)=(p(x),\phi(x))$. 
We 
then have that $r\colon(Y\times G,S_\sigma)\to(Y,S)$ is a factor map and that $q_G\circ r=p$. We get an intermediate extension.

Next, consider two ergodic $G$-cocycles $\sigma,\sigma'\in\CC(Y,G)$ such that $\sigma\circ p$ and $\sigma'\circ p$ are $G$-coboundaries of $(X,T)$. From Lemma~\ref{lem:equi} it is easy to deduce that the associated intermediate extensions are isomorphic if and only if the cocycles $\sigma$ and $\sigma'$ are equivalent.

For an ergodic $G$-cocycle $\sigma\in\CC(Y,G)$, $\alpha\circ p$ is a $G$-coboundary on $(X,T)$ if and only if the class of $\sigma$ in $K_0(Y,S)\times G$ belongs to the kernel of $p^*\otimes\id _G\colon 
 K_0(Y,S)\otimes G\to K_0(X,T)$. We summarize these observations in the next Proposition. The last statement is given without a proof but can be proven by following the same lines.
 
\begin{proposition}
\label{prop:tensor}
Let $p\colon (X,T)\to(Y,S)$ be a factor map between Cantor minimal systems 
and let $G$ be a finite abelian group.

To every $\alpha\in K_0(Y,S)\otimes G$ satisfying
\begin{enumerate}
\item\label{it:ergodic}
$\alpha$ does not belong to $K_0(Y,S)\otimes G'$
for any proper subgroup $G'$ of $G$ and
\item\label{it:coboundary}
$\alpha$ belongs to the kernel of 
$$
p^*\otimes\id _G\colon  K_0(Y,S)\otimes G\to K_0(X,T)\otimes G
$$
\end{enumerate}
 we can associate an intermediate extension which is an extension of $(Y,S)$ by $G$. 

Every intermediate extension which is an extension of $(Y,S)$ by $G$ can be obtained in this way.

Moreover, the intermediate extensions associated to two elements $\alpha,\alpha'$ of  $K_0(Y,S) \allowbreak \otimes G$ satisfying these conditions are isomorphic if and only if there exists an automorphism $\phi$ of $G$ such that $\sigma'=(\id_{K_0(Y,S)}\otimes\phi)(\alpha)$.

Let $G,H$ be two finite groups, and suppose that $\alpha\in K_0(Y,S)\otimes G$ and $\beta\in K_0(Y,S)\otimes H$ satisfy the conditions~\ref{it:ergodic} and~\ref{it:coboundary} above. Then the intermediate extension associated to $\alpha$ is above the intermediate extension associated to $\beta$ if and only if there exists a group homomorphism $\theta$ of $G$ onto $H$ such that $\beta= (\id_{K_0(Y,S)}\otimes\theta)(\alpha)$.
\end{proposition}

\subsection{A translation into an algebraic problem}
The proof of Theorem~\ref{th:finite-abel} now reduces to the proof of the next algebraic result, 
applied with $K=K_0(Y,S)$, $j=p^*$ and $L=K_0(X,T)$.
\begin{proposition}
\label{prop:algebra}
Let
\begin{equation}
\label{eq:exactthree}0\to K\overset{j}{\rightarrow} L \overset{q}{\rightarrow}M\to 0
\end{equation}
be an exact sequence of abelian groups and assume that $L$ is torsion free. 
For every finite abelian group $G$, 
there exists an isomorphism 
$$\Phi_G\colon\ker(j\otimes\id_G)\to \hom(\widehat G,M)$$
where $\widehat G$ is the dual group of $G$.
Moreover, this isomorphism depends 
on the group $G$ in a functorial (covariant) way.

Let $\alpha\in\ker(j\otimes\id_G)$.  The homomorphism
$\Phi_G(\alpha)$ is 
one-to-one if and only if $\alpha$ 
does not belong to
$K\otimes G'$ for any proper subgroup $G'$ of $G$. 

Let $\alpha,\beta\in\ker(j\otimes\id_G)$ be such that
$\Phi_G(\alpha)$ and $\Phi_G(\beta)$ are 
one-to-one. These homomorphisms
have the same range if and only if there exists an automorphism $\phi$ of
$G$ such that
$(\id_K\otimes\phi)(\alpha)=\beta$.
\end{proposition}

This Proposition is certainly well known. We give a proof for completeness but, 
as it is purely algebraic, 
we relegate it to Appendix~\ref{ap:algebra}.

\section{Higher order groups}\label{sec:cohomology}
\subsection{A classical exact sequence in topology}
In this section, $X$ and $Y$ are Cantor spaces and the map $p\colon X\to Y$ is continuous and onto. 
\begin{definition}
For every $n\geq 1$, the $n^{\text{th}}$ relative Cartesian power of $X$ with respect to $p$ is
 $$
X^n_p:=\bigl\{(x_1,x_2,\dots,x_n)\in X^n\colon p(x_1)=p(x_2)=\dots=p(x_n)\bigr\}\ .
$$
The natural projection $X^n_p\to Y$ is written $p_n$.
Thus we have $X^1_p=X$ and $p_1=p$. By convention, $X^0_p=Y$.
\end{definition}
It will be convenient to use this notation also in the case that $Y$ consists of a single point. In this case, we have $X^n_p=X^n$ for every $n\geq 1$ and $X^0_p=\{\text{pt}\}$. 
\begin{definition}
For every $n\geq 1$ and every $f\in\CC(X^n_p)$, let $\partial_nf\in\CC(X^{n+1}_p)$ be the function defined by
\begin{equation}
\label{eq:simplicial}
\partial_nf(x_1,x_2,\dots,x_{n+1})=\sum_{j=1}^{n+1}(-1)^{j+1}
f(x_1,x_2,\dots,\widehat{x_j},\dots,x_{n+1})
\end{equation}
where the symbol $\widehat{x_j}$ means that $x_j$ is to be omitted.
It will be convenient to write $\partial_0=p^*\colon\CC(Y)\to\CC(X)$.
\end{definition}
The next result is
classical, at least in the case that $Y$ consists of a single point. 

\begin{proposition}
\label{prop:exact}
The sequence
\begin{equation}
\label{eq:exact}
0\to\CC(Y)\overset{\partial_0=p^*}{\longrightarrow}\CC(X)
\overset{\partial_1}{\rightarrow}\CC(X^2_p)
\overset{\partial_2}{\rightarrow}\CC(X^3_p)
\overset{\partial_3}{\rightarrow}\cdots
\end{equation}
is exact.
\end{proposition}
The Proposition follows from the next Lemma, 
applied in the case that $Z=Y$. 
We write it in a general form for a possible further use.
\begin{lemma}
\label{lem:factorize}
Let $X,Y$ and $Z$ be Cantor spaces and $r\colon X\to Z$ and $q\colon Z\to Y$ be continuous and onto maps. 
We write $ p= q\circ r$. For every $n\geq 1$, 
we remark that $r^{\times n}:=r\times r\times\dots\times r$ maps $X^n_ p$ 
onto $Z^n_q$.

Let  $n\geq 1$ be an integer and $f\in\CC(X_ p^n)$ be such that $\partial_nf$
factorizes through $r^{\times(n+1)}$. 
Then there 
exist $g\in \CC(X_ p^{n-1})$ and
$h\in\CC(Z^n_q)$ such that $f=\partial_{n-1}g+h\circ r^{\times n}$.
\end{lemma}

\begin{proof}
We consider first the case that $n=1$. In this case we build a function $h\in\CC(Z)$ with 
$f=h\circ r=h\circ r^{\times 1}$.
We have $\partial_1f=u\circ r^{\times 2}$ for some
$u\in\CC(Z^2_q)$. The restriction of $\partial_1f$ to the diagonal of $X^2_ p$ is equal to $0$ and thus, since $r$ is onto, the restriction of $u$ to the diagonal of $Z^2_q$ is equal to $0$.  If
$x,x'\in X$ are such that $ r(x)=r(x')$, then we have $f(x)-f(x')=u(r(x),r(x'))=0$;
it follows that $f$ factorizes through $r$, that is, $f=h\circ r$ for some $h\in\CC(Z)$.

\medskip We assume now that $n\geq 2$.  We claim that for every $y\in Y$ there exist a clopen neighborhood $U$ of $y$ and functions $g\in \CC(X_r^{n-1})$ and $h\in\CC(Z^n_q)$ such that $f=\partial_{n-1}g+h\circ r^{\times n}$ on 
$ p_n\inv(U)$.

Set
$F= p^{-1}\{y\}$ and  $K=q^{-1}\{y\}$. We remark that
$ p_n^{-1}\{y\}=F^n\subset X^n_ p$,
 that $q_n^{-1}\{y\}=K^n\subset Z^n_q$ and that $r^{\times n}$ maps $F^n$ onto $K^n$.

We choose an arbitrary
$a\in F$. The function $F^n\to\Z$ given by $(x_1,\ldots,x_n)\mapsto
(-1)^n\partial_nf(x_1,\ldots,x_n,a)$ factorizes through $K^n$
by hypothesis: there exists $v\in\CC(K^n)$ such that, for $(x_1,\dots,x_n)\in F^n$, we have
$$
(-1)^n\partial_nf(x_1,\ldots,x_n,a)=v\circ r^{\times n}(x_1,\dots,x_n)\ .
$$
For $(x_1,\ldots,x_{n-1})\in F^{n-1}$, we define
 $$
 u(x_1,\ldots,x_{n-1})=(-1)^{n+1}f(x_1,\ldots,x_{n-1},a)\ .
 $$
An immediate computation gives
$$\partial_{n-1}u +v\circ r^{\times n}=f\text{ on }F^n\ .
$$

Since $u$ is integer valued and continuous on the closed subset $F^{n-1}$ of the Cantor set $X^{n-1}_ p$,
it admits a continuous extension  $g\in\CC(X^{n-1}_ p)$ to  $X^{n-1}_ p$.
In the same way, the function $v\in\CC(K^n)$ admits a continuous extension 
$h\in  \CC(Z^n_q)$.
 
As $\partial_{n-1}g+h \circ r_n=f$ on $ p_n^{-1}\{y\}$ and all
these functions are continuous and integer valued and thus locally constant, 
this equality holds on some
neighborhood of $ p_n^{-1}\{y\}$ in $X^n_ p$.
Thus on $ p_n^{-1}(U)$ for some
clopen neighborhood  $U$ of $y$.  Our claim is proved.

As $Y$ is Cantor, we can
find a finite partition of $Y$ by clopen sets constructed as above, and by gluing together the different functions, we get  functions $g$ and
$h$ with the required properties.
\end{proof}

\begin{proof}[Proof of Proposition~\ref{prop:exact}]
Clearly, $\partial_{n+1}\circ\partial_n=0$ for all $n\geq 0$. 

Let $n\geq 1$ be an integer and $f\in\CC(X^{n-1}_p)$ belong to $\ker(\partial_n)$. We build a function $w\in\CC(X^{n-1}_p)$ with $f=\partial_{n-1}w$.

We use Lemma~\ref{lem:factorize} applied with $Z=Y$ and $q=\id$, noticing that $Z^n_q=Y$ and that $r^{\times n}=p_n$. We get that there exist $g\in\CC(X^{n-1}_p)$
and $h\in\CC(Y)$ with 
$$
f=\partial_{n-1}g+h\circ p_n\ .
$$
If $n$ is odd, we have $\partial_{n-1}(h\circ p_{n-1})=h\circ p_n$ and thus
$f=\partial_{n-1}(g+h\circ p_{n-1})$ and we can take $w=g+h\circ p_{n-1}$. If $n$ is even we have 
$$
0=\partial_nf=\partial_n\circ \partial_{n-1}g+\partial_n(h\circ p_n)=\partial_n(h\circ p_n)=
h\circ p_{n+1}
$$
and thus $h=0$, $f=\partial_{n-1}g$ and we can take $w=g$.
\end{proof}
\subsection{Higher order cohomology groups of an extension}
In the rest of this Section we assume that $p\colon(X,T)\to(Y,S)$ is a factor map between Cantor minimal dynamical systems.

For every $n\geq 1$, let $X^n_p$ be endowed with the diagonal transformation $T^{\times n}=T\times T\times\dots\times T$. In particular, $T^{\times 1}=T$ and $T^{\times 0}=S$ by convention. We write also $\beta_n$ instead of 
$\beta_{T^{\times n}}\colon\CC(X^n_p)\to \CC(X^n_p)$.

For every $n\geq 0$, we have clearly 
\begin{equation}
\label{eq:beta-partial}
\beta_{n+1}\circ\partial_n=\partial_n\circ\beta_n
\end{equation}
and thus the group homomorphism $\partial_n\colon \CC(X^n_p)\to \CC(X^{n+1}_p)$
 induces a group homomorphism
$$
\partial_n^*\colon K_0(X^n_p,T^{\times n})=
\frac{\CC(X^n_p)}{\beta_n\CC(X^n_p)}\longrightarrow
K_0(X^{n+1}_p,T^{\times (n+1)})=
\frac{\CC(X^{n+1}_p)}{\beta_{n+1}\CC(X^{n+1}_p)}\ .
$$
The exact sequence~\eqref{eq:exact} induces a sequence
\begin{equation}
\label{eq:inexact}
0\to K_0(Y,S)\overset{\partial_0^*=p^*}{\longrightarrow}   K_0(X,T)
\overset{\partial_1 *}{\rightarrow}  K_0(X^2_p,T^{\times 2})
\overset{\partial_2^*}{\rightarrow}K_0(X^3_p,T^{\times 3})
\overset{\partial_3^*}{\rightarrow}\cdots
\end{equation}
We have $\partial_{n+1}^*\circ\partial_n^*=0$ for every $n$.
\begin{definition}
We define the cohomology groups of the extension $p\colon(X,T)\to(Y,S)$ to be the homology groups of the sequence~\eqref{eq:inexact}. More explicitly, for $n\geq 0$,
$$
H^n(X\mid Y):=\frac{\ker(\partial_{n+1}^*)}{\range(\partial_n^*)}\ .
$$
\end{definition}

\subsection{Comparing $H^0(X\mid Y)$ and $K_0(X,T)/p^*K_0(Y,S)$}
In particular we have that 
$$
H^0(X\mid Y):=\frac{\ker(\partial_{1}^*)}{p^*K_0(Y,S)}
$$
is a subgroup of the group $K_0(X,T)/p^*K_0(Y,S)$ considered in Section~\ref{sec:abelian},
in which we gave an interpretation of the torsion subgroup
of this last group. In fact, because $K_0(X^2_p,T^{\times 2})$ is torsion free, we have:
\begin{corollary}
\label{cor:torsion}
$$
\torsion\Bigl(\frac{K_0(X,T)}{p^*K_0(Y,S)}\Bigr)
=\torsion\bigl(H^0(X\mid Y)\bigr)\ .
$$
\end{corollary}
Below we show that $H^0(X\mid Y)$ is a torsion group (Theorem~\ref{thm:H=torsion}).
\begin{proof}
Suppose $f \in \mathcal{C}(X)$ represents an element of
order $k$ in $\frac{K_0(X,T)}{p^*K_0(Y,S)},\ k [f] =0$; i.e.
\begin{equation}\label{kf}
kf(x) = g(y) + \phi(x)  - \phi(Tx),
\end{equation} 
for some $g \in \mathcal{C}(Y)$ and $\phi \in \mathcal{C}(X)$. 
We have to show that 
$f \in \ker(\partial^*_1)$, that is, that there exists $\psi \in \mathcal{C}(X^2_p)$ with $\partial_1(f)(x_1,x_2) = f(x_2) - f(x_1) = \psi(x_1,x_2) - \psi(Tx_1,Tx_2)$ on $X^2_p$.
Now from (\ref{kf}) we deduce that
$$
k(f(x_1) - f(x_2)) = \phi(x_1) - \phi(x_2) - (\phi(Tx_1) - \phi(Tx_2))=
\partial_1\phi \circ T^{\times 2} - \partial_1 \phi.
$$
This means that $k[f(x_1) - f(x_2)] = 0$ in $K_0(X^2_p,T^{\times 2})$.
However, the latter group is torsion free and we conclude that also
$[f(x_1) - f(x_2)]=0$, as required. 
\end{proof}

We immediately deduce the following: 
\begin{corollary}
Theorems~\ref{th:finite-abel} and~\ref{th:abel-max} remain valid with 
$H^0(X\mid Y)$ replacing 
\allowbreak 
$K_0(X,T)\allowbreak /p^*K_0(Y,S)$.
\end{corollary}

\begin{remark}\label{alln}
Essentially the same proof yields the following equation for all $n$, 
$$
\torsion\Bigl(\frac{K_0(X^{n+1}_p,T^{\times (n+1)})}
{\partial_n^*( K_0(X^n_p,T^{\times n}))}\Bigr)
=\torsion\bigl(H^n(X\mid Y)\bigr)
$$
and below we show also that $H^n(X\mid Y)$ is a torsion group (Theorem~\ref{thm:H=torsion}).
\end{remark}

It seems that in some cases 
the group $H^0(X\mid Y)$ is easier to compute than  the quotient 
$K_0(X,T)/p^*K_0(Y,S)$.  $H^0(X\mid Y)$ gives interesting  information about the extension, but some properties can not be seen on this group.

For example, consider the case that some $S$-invariant probability measure 
$\nu$ on $(Y,S)$ can be lifted as two distinct
$T$-invariant
probability measures $\mu$ and $\mu'$ on $(X,T)$. Then we can build the invariant probability $\omega=\mu\times_p\mu'$ on $X^2_p$. Let $f\in\CC(X)$ be such that $\int f\,d\mu\neq\int f\,d\mu'$. Then 
$$\int \bigl(f(x_1)-f(x_2)\bigr)\,d\omega(x_1,x_2)\neq 0
$$ and thus the function $\partial_1f$ is not a coboundary in $\CC(X^2_p,T^{\times 2})$. We constructed an element $[f]$ in $K_0(X,T)$ with 
$\partial_1[f]\neq 0$ and thus
$H^0(X\mid Y)\subsetneq K_0(X,T)/p^*K_0(Y,S)$. More generally, it seems that the group $H^0(X\mid Y)$ does not give any information on the question of unique lifting of measures.
 
\subsection{Cohomology groups via invariants}
\label{subsec:cohoInvar}
\begin{definition}
For every $n\geq 0$ we write $\CI(X^n_p)$ for the subgroup of $\CC(X^n_p)$ consisting of functions invariant under $T^{\times n}$.
\end{definition}
Recall that by convention $(X^0_p,T^{\times 0})=(Y,S)$. By minimality, all invariant continuous functions on $Y$ are constant and $\CI(X^0_p) = \CI(Y)=\Z$.
We have also  $(X^1_p,T^{\times 1})=(X,T)$ and thus $\CI(X^1_p) = \CI(X)=\Z$.

We remark that $\CI(X^n_p)$ is the kernel of the coboundary map 
$\beta_n\colon\CC(X^n_p)\to\CC(X^n_p)$. By~\eqref{eq:beta-partial}, for every $n\geq 0$ we have 
$\partial_n\CI(X^n_p)\subset \CI(X^{n+1}_p)$.
 We write 
$d_n\colon \CI(X^n_p)\to\CI(X^{n+1}_p)$ for the restriction of $\partial_n$ to 
$\CI(X^n_p)$.

\begin{proposition}\label{prop:chasing}
For every $n\geq 0$, 
$$
H^n(X\mid Y)\cong\frac {\ker(d_{n+3})}{\range(d_{n+2})} \ .
$$
Moreover, $d_2\colon \CI(X^2_p)\to\CI(X^3_p)$ is 
one-to-one.
\end{proposition}
\begin{proof}
We write $j_n\colon \CI(X^n_p)\to\CC(X^n_p)$ for the inclusion map.

We get the commutative diagram:
\begin{equation}
\label{eq:diagram}
\begin{matrix}
 & & 0 & & 0 &  & 0 & & 0 &\\
 & & \downarrow & & \downarrow & & \downarrow & & \downarrow &\\
    0 & \to & \CI(Y) & \overset{d_0}{\rightarrow} & \CI(X)
    & \overset{d_1}{\rightarrow} & \CI(X^2_p)
    & \overset{d_2}{\rightarrow} & \CI(X^3_p)
    &\overset{d_3}{\rightarrow}\cdots\\
 & & \downarrow j_0 & & \downarrow j_1 & & \downarrow j_2 & & \downarrow j_3 &\\
   0 & \to & \CC(Y) & \overset{\partial_0}{\rightarrow} & \CC(X)
   & \overset{\partial_1}{\rightarrow} & \CC(X^2_p)
   & \overset{\partial_2}{\rightarrow} & \CC(X^3_p)
   &\overset{\partial_3}{\rightarrow}\cdots\\
 & & \downarrow\beta_0 & & \downarrow\beta_1 & & \downarrow\beta_2 & & \downarrow\beta_3 &\\
 0 & \to & \CC(Y) & \overset{\partial_0}{\rightarrow} & \CC(X)
   & \overset{\partial_1}{\rightarrow} & \CC(X^2_p)
   & \overset{\partial_2}{\rightarrow} & \CC(X^3_p)
   &\overset{\partial_3}{\rightarrow}\cdots\\
 & & \downarrow & & \downarrow & & \downarrow & & \downarrow &\\
 0 & \to & K_0(Y,S) & \overset{\partial_0^*}{\rightarrow} & K_0(X,T)
   & \overset{\partial_1^*}{\rightarrow} & K_0(X^2_p,T^{\times 2})
   & \overset{\partial_2^*}{\rightarrow} & K_0(X^3_p,T^{\times 3})
   &\overset{\partial_3^*}{\rightarrow}\cdots\\
 & & \downarrow & & \downarrow & & \downarrow & & \downarrow &\\
 & & 0 & & 0 &  & 0 & & 0 &
\end{matrix}
\end{equation}
where the vertical arrows $\CC(X^n_p)\to K_0(X^n_p,T^{\times n})$ are the quotient maps.
Therefore  the columns of this diagram are exact sequences.
In the first row, $\CI(Y)=\CI(X)=\Z$ and $d_0$ is the identity map. 
By construction, $d_1$ is the zero map.

By Proposition~\ref{prop:exact}, the second
and third rows of this 
diagram are exact.
We get the announced result by diagram chasing.
\end{proof}

\begin{comments}
Proposition \ref{prop:chasing} turns out to be a powerful tool in computing cohomology groups. The main reason for this is as follows. 
For a general dynamical system $(X,T)$ one defines the {\em prolongation relation} on $X$ 
(see \cite{AS}) :
\begin{definition}
${\rm Prol}(x)$ is the set of pairs $(x,x')\in X^2$ such that there exist a sequence $(x_i)$ converging to $x$ in $X$ and a sequence 
$(n_i)$ of integers such that the sequence $(T^{n_i}x_i)$ converges to $x'$.
\end{definition}
It is now easy to see that a continuous invariant function on $X$
is necessarily constant on each prolongation class 
${\rm Prol}[x]$.
Note also that for every $x$, its orbit closure is contained in ${\rm Prol}[x]$.
These facts will be used repeatedly in the sequel. For more details
on the prolongation relation see \cite{AG}.
\end{comments}

\section{Some particular cases}

\subsection{Proximal and weakly mixing extensions}
Here and in the sequel, $d$ denotes a distance defining the topology of $X$.

\begin{definition}
\label{def:extension}
Let $p\colon(X,T)\to(Y,S)$ be an extension between Cantor minimal systems.
\begin{enumerate}
\item
We say that $p$ is a \emph{proximal extension} if for every $x_1,x_2\in X$ with 
$p(x_1)=p(x_2)$
we have $\inf_{k>0} d(T^kx_1,T^kx_2)=0$.
\item
It is a \emph{weakly mixing extension} if the system $X_p^2$ is topologically ergodic 
(or topologically transitive), that is, there is a point $(x_1,x_2) \in X_p^2$
with dense $T^{\times 2}$-orbit. 
\item
It is a \emph{weakly mixing of all orders} extension if $X_p^n$ is topologically ergodic for 
every $n \ge 1$.
\item
When $Y$ is a trivial one point system we recover the usual definitions 
of a weak mixing system and a weakly mixing system of all orders. 
By a theorem of Furstenberg, a weakly mixing system is also 
weakly mixing of all orders.
\end{enumerate}
\end{definition}

\begin{theorem}\label{thm:proximal}
\label{th:proximal}
If $p\colon(X,T)\to(Y,S)$ is either a proximal extension
or a weakly mixing of all orders extension between Cantor minimal systems then $H^n(X\mid Y)=0$ for every $n\geq 0$.
In particular, for a minimal Cantor weakly mixing system $(X,T)$ we have
$H^n(X) = 0$ for every $n \ge 0$.
\end{theorem}

\begin{proof}
Consider first the case when $p$ is a proximal extension.
We claim that for every $n\geq 0$, all the $T^{\times n}$-invariant functions on $X^n_p$ are constant, that is, $\CI(X^n_p)=\Z$. 

For $n=0$ or $1$ the claim is obvious by minimality. 
Let $n\geq 2$ and $f\in\CI(X^n_p)$. 
Recall that $f$ is integer valued and continuous, thus locally constant. By minimality, $f$ is equal to a constant $c\in\Z$ on the diagonal of $X^n_p$.
 Let $\bx=(x_1,\dots,x_n)\in X^n_p$. 
By induction on $n$, we have that 
$\inf_{k>0}\diameter(\{T^kx_1,\dots,T^kx_n\})=0$. 
Therefore,
 there exists $k>0$ such that 
$f(T^kx_1,\dots,T^kx_n)=f(T^kx_1,\dots,T^kx_1)=c$ and by invariance we get that $f(\bx)=c$. The claim is proved.

Therefore, in the first row of the diagram~\eqref{eq:diagram}, all the groups are equal to $\Z$; By definition, $d_n$ is the identity map if $n$ is even and is the zero map if $n$ is odd. Therefore, the first row of the diagram is an exact sequence. The announced result follows now from Proposition~\ref{prop:chasing}.

The proof when $p$ is assumed to be weakly mixing of all orders is similar
and straightforward.
If $f \in \mathcal{I}^n_p$ and $(x_1,x_2,\dots,x_n)$ is a point of $X_p^n$ with dense orbit then 
for every $k,l$ we have $f(T^kx_1,\dots,T^kx_n)=f
(T^lx_1,\dots,T^lx_n)=c$
and by the continuity of $f$ we conclude that $f = c$.
\end{proof}

In particular, we have in this cases that $H^0(X\mid Y)=0$. 
It now follows that for any weakly mixing minimal Cantor system $(X,T)$
we have $0 = 
H^0(X) \subsetneq K_0(X,T)$.

The group 
$K_0(X,T)/p^*K_0(Y,S)$ can not be computed so easily. 
Since the sequence~\eqref{eq:inexact} is exact, we have:
$$
\frac{K_0(X,T)}{p^*K_0(Y,S)}\cong \range(\partial_1^*)=\ker(\partial_2^*)\ .
$$

We next investigate intermediate extensions. We need a Lemma.
\begin{lemma}
\label{lem:proximal}
Let $r\colon (X,T)\to(Z,R)$ be a proximal extension. Then, for every $\epsilon>0$, every integer $n>0$ and all $(x_1,x'_1)$, \dots, $(x_n,x'_n)$ in $X^2_r$ there exists an integer $k$ such that
$d(T^kx_1,T^kx'_1)<\epsilon$, \dots, $d(T^k_n,T^kx'_n)<\epsilon$.
\end{lemma}

\begin{proof} 
The proof goes by induction on $n$. For $n=1$, the statement is nothing else than the definition of a proximal extension. We take $n\geq 2$ and assume that the results holds for $n-1$ pairs in $X^2_r$.

Let $(x_1,x'_1)$, \dots, $(x_n,x'_n)\in X^2_r$. By the induction hypothesis, there exists a sequence $(k_j)$ of integers such that, for $1\leq i\leq n-1$, $d(T^{k_j}x_i,T^{k_j}x'_i)\to 0$ as $j\to+\infty$ . Substituting a subsequence for this sequence, we can assume that for $1\leq i\leq n-1$ the sequences 
$(T^{k_j}x_i)$ and $(T^{k_j}x'_i)$ converge to the same point $y_i\in X$. We can assume also that the sequence $(T^{k_j}x_n)$ converges to some point $y_n$ and the sequence $(T^{k_j}x'_n)$ to some point $y'_n$. Since $(x_n,x'_n)\in X^2_r$, we have that $(y_n,y'_n)\in X^2_r$.

Since the extension $r$ is proximal, there exists an integer $\ell$ such that $d(T^\ell y_n,T^\ell y'_n)\allowbreak <\epsilon/3$. By continuity of $T^\ell$, for $j$ sufficiently large we have
$d(T^\ell y_n,T^{\ell+k_j}x_n)<\epsilon/3$, $d(T^\ell y'_n,T^{\ell+k_j}x'_n)<\epsilon/3$   and thus 
$d(T^{\ell+k_j}x_n, T^{\ell+k_j}x'_n)<\epsilon$. On the other hand, by continuity of $T^\ell$ again,
for  $j$ sufficiently large we have $d(T^{\ell+k_j}x_i, T^\ell y_i)<\epsilon/2$ and 
$d(T^{\ell+k_j}x'_i, T^\ell y_i)<\epsilon/2$, and thus $d(T^{\ell+k_j}x_i,T^{\ell+k_j}x'_i)<\epsilon$. 
We get the announced statement with $k=\ell+k_j$ for $j$ sufficiently large.
\end{proof}

\begin{proposition}
\label{HnXZ}
Let $p : (X,T) \to (Y,S)$ be an extension between Cantor minimal systems and 
\begin{equation}
\label{eq:tri}
\xymatrix{
(X,T) \ar[dr]^r \ar[rr]^p & & (Y,S)\\
 & (Z,R) \ar[ur]^q & 
 }
\end{equation}
an intermediate extension with $Z$ Cantor, such that $r$ is 
a proximal extension. Then 
$$
H^n(X \mid Y) = H^n(Z \mid Y),
$$
for every $n \ge 0$.
\end{proposition}

\begin{proof}

The map $r^{\times n}\colon (X_p^n,T^{\times n})\to (Z_q^n,R^{\times n})$ is a factor map, and thus the map $r^{\times n*}\colon f\mapsto f\circ r^{\times n}$
$$
r^{\times n*}(f)(x_1,\dots,x_n)=f(r(x_1),\dots,f(r(x_n))
$$
maps $\CI(Z^n_q)$ into $\CI(X^n_p)$ in a 
one-to-one way. By Proposition~\ref{prop:chasing} it suffices to prove that this map is onto. 
This claim is equivalent to:
\begin{itemize}
\item[]\it 
For every $f\in\CI(X^n_p)$ and for every $\bx=(x_1,\dots,x_n)$ and every $\bx'=(x'_1,\dots,x'_n)$ in $X^n_p$, if 
\begin{equation}
\label{eq:hyp2}
r(x_1)=r(x'_1),\ r(x_2)=r(x'_2),\ \dots,\text{ and }r(x_n)=r(x'_n)\ .
\end{equation}
then $f(\bx)=f(\bx')$.
\end{itemize}
Assume that $\bx$ and  $\bx'\in X^n_p$  satisfy~\eqref{eq:hyp2}. 
Since $f$ is locally constant, there exists $\epsilon>0$ such that, if $\by$ and $\by'\in X^n_p$ satisfy
$d(y_i,y'_i)<\epsilon$ for $1\leq i\leq n$, then $f(\by)=f(\by')$.
Since the extension $r$ is proximal, by Lemma~\ref{lem:proximal} there exists 
 an integer $k$ with
$$
d(T^k x_i,T^kx'_i)<\epsilon\text{ for } 1\leq i\leq n\ .
$$
Therefore  $f(T^kx_1,\dots,T^kx_n)=f(T^kx'_1,\dots,T^kx'_n)$. Since $f$ is invariant, we get $f(\bx)=f(\bx')$. 
\end{proof}

\subsection{Extensions with connected 
fibers}

\begin{theorem}\label{Hnconnected}
Let $p: (X,T) \to (Y,S)$ be an extension between Cantor minimal systems.
Consider an intermediate extension, as in (\ref{eq:tri}),
such that the fibers of $q$ are connected and $r$ is a proximal extension.
Then $H^n(X \mid Y)=0$ for every $n \ge 0$. 
\end{theorem}
This result applies
in the case that $p$ is a compact connected group extension, or more generally 
an isometric extension with connected fibers.
In particular, taking $Y$ to be the trivial one point system,
we conclude that $H^n(X)=0$ for every Cantor minimal system $(X,T)$
which is an almost one-to-one extension (or more generally a proximal
extension) of a compact connected monothetic group.

\begin{proof}
In fact we will show that  
$\CI(X_p^n) = \Z$ for every $n \geq 0$.
Since, by Proposition \ref{prop:chasing}, $H^n(X\mid Y)\cong
{\ker(d_{n+3})}/{\range(d_{n+2})}$, this is an even stronger result.

By exactly the same method as in the proof of Proposition~\ref{HnXZ}, the map 
$r^{\times n*}\colon\allowbreak\CI(Z^n_q)\to\CI(X^n_p)$ is bijective. 

 On the other hand, all the fibers of the projection $q_n\colon Z^n_\sigma\to Y$ are connected because $q_n^{-1}\{y\}=(q^{-1}\{y\})^n$ for every $y\in Y$.
  Since $f$ is locally constant, it is constant on each fiber  $q_n^{-1}\{y\}$, and thus $f=g\circ q_n$ for some function $g\in\CC(Y)$.
 Since $f$ is invariant under $R^{\times n}$, $g$ is invariant under $S$ and, since $(Y,S)$ is minimal,
 $g$ is constant. Therefore, $f$ is constant.
 
 We conclude that every function in $\CI(X^n_p)$ is constant, and this group is equal to $\Z$.
\end{proof}

\section{The case of a finite or Cantor isometric extension}
\label{sec:isometric}
\subsection{Preliminaries}

The notion of an isometric extension  was introduced by
Furstenberg \cite{F}. It was shown there that an extension 
$p : (X,T) \to (Y,S)$ between minimal systems is isometric if and only if
it is an equicontinuous extension. In the case where $X$ and $Y$ are
a Cantor minimal systems it can be shown that an isometric extension has the following form:

\begin{definition}
\label{def:isom}
Let $(Y,S)$ be a Cantor minimal system, $G$ a finite (resp. Cantor)  group, $H$ a closed subgroup, $K=G/H$ endowed with the quotient topology and the left action of $G$, and $\sigma\colon Y\to G$ a continuous map, called a \emph{cocycle}. Then $X=Y\times K$ endowed with the transformation $T\colon (y,u)\mapsto (Sy\sigma(y)\cdot u)$ and with the projection $p\colon (y,u)\mapsto y$ on $Y$, is called a \emph{finite (resp. Cantor) isometric extension} of $(Y,S)$.
\end{definition}

Before we begin our study we introduce some reductions.  
First, the compact group $G$ admits a translation invariant distance, inducing a distance on $K=G/H$ which is invariant under the left action of $G$.
Substituting a quotient of $G$ for $G$ if needed, we can assume that $H$ does not contain any proper normal subgroup of $G$. Then $G$ can be considered as a subgroup of the compact group of isometries of the compact metric space $K$ and $H$ is the stabilizer of some element of $K$. 

This leads to an equivalent definition of a finite or Cantor isometric extension. 
We start with a finite or a Cantor space
 $K$, define $X=Y\times K$, and let $p\colon X\to Y$ to be the first projection. 
 Let $T\colon X\to X$ be a homeomorphism such that $(X,T)$ is minimal and  that $p\colon (X,T)\to(Y,S)$ is a factor map. Then $T$ has the form $T(y,u)=(Sy,\sigma(y)(u))$ where $\sigma(y)$ is a homeomorphism from $K$ onto itself. 
If $\sigma(y)$ is an isometry of $K$ for every $y\in Y$, then the extension $p$ is an isometric extension. Indeed, we define 
 $G$ to be the group of isometries of $K$, endowed with the topology of uniform convergence. It is a finite or a Cantor group. The minimality of $(X,T)$ implies that $G$ acts transitively on $K$. Then, setting $H$ to be the stabilizer of some point of $K$, we can identify $K$ with $G/H$ and let $(X,T)$ be the isometric extension defined as above.

For simplicity we restrict ourselves now to the case of finite extensions, and explain shortly later the changes needed for the Cantor case.
\begin{remark}
Let $p\colon (X,T)\to(Y,S)$ be an extension between Cantor minimal systems. Then it can be shown that it is a finite isometric extension if and only if the map $p$ is finite to one and open. We do not use this fact in the sequel.
\end{remark}

Let $p\colon  (X,T)\to(Y,S)$ be a finite isometric extension between  Cantor minimal systems and let $G,H,K,\sigma$ be  as in Definition~\ref{def:isom}.

We will replace the group $G$ by the \emph{Mackey group} of the cocycle (see~\cite{LM} and~\cite{FW2}),
whose definition we next recall. 
For every  integer $k$, the transformation $T^k$ of $X$ is given by
$$
T^k(y,u)=(S^ky,\sigma^{(k)}(y)\cdot u)$$
where $\sigma^{(0)}=1$,
$$
\sigma^{(k)}(y)=\sigma(S^{k-1}y).\cdots.\sigma(Sy).\sigma(y)\text{ for }k>0,\
$$
and a similar formula for $k<0$. Fix a point $y_0\in Y$ and let $M$ be the set of limit points of the sequences $(\sigma^{(k_i)}(y_0)\colon i\geq 1)$ for all sequences $(k_i\colon i\geq 1)$ of integers such that $S^{k_i}y_0$ converges to $y_0$. Then $M$ is a subgroup of $G$, called the Mackey group of $\sigma$. Although this definition of $M$ depends on $y_0$, it can be shown that the Mackey groups corresponding to different points of $Y$ are conjugate. The minimality of $(X,T)$ implies that $M$ acts transitively on $K$. Moreover, the cocycle $\sigma$ is cohomologous to an $M$-valued cocycle $\tau$. This means that there exist a continuous map $\tau\colon Y\to M$ and a continuous map $\phi\colon Y\to G$, such that
 $$
\tau(y)=\phi(Sy)^{-1}\,\sigma(y)\,\phi(y)\text{ for every }y\in Y\ .
$$
By an obvious change of variables, we can replace the cocycle $\sigma$ by $\tau$, that is, we can assume that $\sigma$ takes its values in the Mackey group $M$. Substituting $M$ for $G$ and $M\cap H$ for $H$, we reduce to the case that $G$ is equal to $M$. From the definition of this group, we get
\begin{itemize}
\item[(*)]
\it
For every $y\in Y$, every neighborhood $U$ of $y$ in $Y$, and every $g\in G$,  there exists an integer $k$ such that $T^ky\in U$ and $\sigma^{(k)}(y)=g$.
\end{itemize} 

We consider now the case of a Cantor isometric extension between Cantor minimal systems.  As above, we can assume that $G$ is a closed subgroup of the compact group of isometries of $K$. The Mackey group $M$ is defined as in the finite case, and in this case we can also assume that the cocycle $\sigma$ takes its values in $M$, and  replace $G$ by $M$.
 The property~(*) is replaced by the more general property:
\begin{itemize}
\item[(**)]
\it 
For every $y\in Y$, every neighborhood $U$ of $y$ in $Y$, and every non empty open subset  $W$ of $G$,   there exists an integer $k$ such that $T^ky\in U$ and $\sigma^{(k)}(y)\in W$.
\end{itemize} 

\subsection{Invariants of $X^n_p$ and invariants of $K^n=(G/H)^n$}\strut

In this Section, $p\colon(X,T)\to(Y,S)$ is a finite of Cantor isometric extension between Cantor minimal systems. We use the notation of Definition~\ref{def:isom}, the reductions made in the preceding section, and we further assume that  the property~(**) holds.

For every $n\geq 1$, let $K^n=(G/H)^n$ be endowed with the diagonal left action of $G$:
$$
\text{For }h\in G\text{ and }\bu=(u_1,\dots,u_n)\in K^n,\  
h\cdot\bu=(h\cdot u_1,\cdots,h\cdot u_n)\ .
$$
 We write $\CI(K^n)$ for the subgroup of $\CC(K^n)$ consisting in functions invariant under this action.

Recall that for every $n\geq 1$, $X^n_p$ can be identified with $Y\times K^n$ so that $p_n\colon X^n_p\to Y$ is the first coordinate projection. Under this identification, the transformation $T^{\times n}$ of $X^n_p$ has the form 
$$
(y,\bu)=(y,u_1,\dots,u_n)\mapsto (Sy,\sigma(y)\cdot\bu)=(Sy,\sigma(y)u_1,\dots,\sigma(y)u_n)\ .
$$

\begin{lemma}
\label{lem:invar}
We keep using the same notation.
Let $y\in Y$. 
Then the restriction map $\rho_n\colon \CC(X^n_p)\to \CC(p_n\inv\{y\})$,
when the latter is identified with $\CC(K^n)$, induces a bijection from $\CI(X^n_p)$ onto $\CI(K^n)$.
\end{lemma}

\begin{proof}

First we show that $\rho_n$ maps $\CI(X^n_p)$ to $\CI(K^n)$. Let $f\in\CI(X^n_p)$ and $h\in G$.
Since $f$ is locally constant, there exists a neighborhood $U$ of $y$  and a neighborhood $W$ of the unit $e_G$ of $G$ such that
$f(z,w\cdot\bu)=f(y, \bu)$ for every $z\in U$, every $\bu\in K^n$  and every $w\in W$.
By the property~(**) above, there exists an integer $k\geq 1$ with 
$S^ky\in U$ and $\sigma^{(k)}(y)\in Wh$.
Chose $k$ with this property and write $\sigma^{(k)}(y)=wh$ with $w\in W$. For every
 $\bu\in K^n$, by invariance of $f$  we have
 \begin{multline*}
f(y,\bu)=f(T^k(y,\bu))=f(S^ky, \sigma^{(k)}(y)\cdot\bu)
=f(S^ky,wh\cdot\bu)=f(y,h\cdot\bu)
\end{multline*}
because $S^ky\in U$ and $w\in W$.
Therefore, the restriction of $f$ to $p_n\inv\{y\}$ belongs to $\CI(K^n)$.

We check now that the restriction map $\rho_n\colon \CI(X^n_p)\to\CI(K^n)$ is onto. Let $\phi\in\CI(K^n)$ and define $f\in\CC(X^n_p)$ by $f(z,\bu)=\phi(\bu)$ for every $(z,\bu)\in X^n_\sigma$. For every $(z,\bu)\in X^n_\sigma$, 
$f(T(z,\bu))=f(Sz,\sigma(z)\cdot\bu)= \phi(\sigma(z)\cdot\bu)=\phi(\bu)=f(z,\bu)$. Therefore $f\in\CI(X^n_p)$ and clearly we have  $\rho_nf=\phi$.

Next let us check that the restriction map $\rho_n\colon \CI(X^n_p)\to\CI(K^n)$ is
one-to-one.
Let $f\in\CI(X^n_p)$ be such that $\rho_nf=0$. This means that the restriction of $f$ to $p_n\inv\{y\}$ is equal to $0$. By invariance, for every $k$ the restriction of $f$ to $p_n\inv\{S^ky\}$ is also equal to $0$. Therefore, $f$ is equal to zero everywhere by minimality and continuity.
\end{proof}

\begin{corollary}
\label{cor:isom}
Let $p\colon (X,T)\to(Y,S)$ be a Cantor or finite isometric extension. As above, we write $X=Y\times K$ with $K=G/H$ and we use the preceding reductions. For every $n$,
let $d_n\colon \CI(U^n)\to\CI(K^{n+1})$ be the restriction to $\CI(K^n)$ of the simplicial differential $\partial_n\colon\CC(K^n)\to\CC(K^{n+1})$ defined by~\eqref{eq:simplicial}.

Then $H^*(X\mid Y)$ is the homology of the chain
\begin{equation}
\label{eq:CI}
0\to\CI(K)\overset{d_1}{\rightarrow}  \CI(K^2) 
 \overset{d_2}{\rightarrow}  \CI(K^3)
    \overset{d_3}{\rightarrow}\cdots
\end{equation}
More precisely,
$$
H^n(X\mid Y)\cong\frac {\ker(d_{n+3})}{\range(d_{n+2})} \ .
$$
\end{corollary}
\begin{proof}[Proof of Corollary~\ref{cor:isom}]
Immediate by Proposition~\ref{prop:chasing} and Lemma~\ref{lem:invar}.
\end{proof}

\subsection{Torsion of the groups $H^n(X\mid Y)$ in the isometric case}
\label{subsec:torsion_isometric}
\begin{theorem}
\label{th:tosion_Isometric}
Let $p\colon (X,T)\to(Y,S)$ be an extension between Cantor minimal systems.

\begin{enumerate}
\item
If $p$ is a finite isometric extension then for all $n$ the group $H^n(X\mid Y)$ is finite.
\item
If $p$ is a Cantor isometric extension then for all $n$ the group $H^n(X\mid Y)$ is a torsion group.
\end{enumerate}
\end{theorem}
\begin{proof}
We continue using the same notation and conventions as above. 
We assume first that the extension is finite and write $K=G/H$, where $G$ is a finite group.

For every $n$, let $d_n\colon \CI(K^n)\to\CI(K^{n+1})$ be defined as in Corollary~\ref{cor:isom}.
We show that $H^n(X\mid Y)\cong \ker(d_{n+3})/\range(d_{n+2})$ is a finite group. 

Let $f\in\ker(d_{n+3})$ and $[f]$ its class in $\ker(d_{n+3})/\range(d_{n+2})$. 
 By Lemma~5.4 there exists $h\in\CC(K^{n+2})$ with $\partial_{n+2}h=f$.
For every $g\in G$ define the function $h_g\in\CC(K^{n+2})$ by 
$$
h_g(u_1,\dots,u_{n+2})=h(g\cdot u_1,\dots,g\cdot u_{n+2})\ .
$$
We have
\begin{align*}
\partial_{n+2}h_g(u_1,\dots,u_{n+3})
&=\sum_{j=1}^{n+3}(-1)^jh_g(u_1,\dots,\widehat{u_j},\dots,u_{n+3})\\
&:=\sum_{j=1}^{n+3}(-1)^jh_g(u_1,\dots,u_{j-1},u_{j+1},\dots,u_{n+3})\\
&=\sum_{j=1}^{n+3}(-1)^jh(g\cdot u_1,\dots,g\cdot u_{j-1},g\cdot u_{j+1},\dots,g\cdot u_{n+3})\\
&=\sum_{j=1}^{n+3}(-1)^jh(g\cdot u_1,\dots,\widehat{g\cdot u_j},\dots,g\cdot u_{n+3})\\
&=\partial_{n+2} h(g\cdot u_1,\dots,g\cdot u_{n+3})\\
&=f(g\cdot u_1,\dots,g\cdot u_{n+3})=
f(u_1,\dots,u_{n+3})
\end{align*}
because $f\in\CI(K^{n+3})$. We define
$$
q(u_1,\dots,u_{n+2}):=\sum_{g\in G} h_g(u_1,\dots,u_{n+2})
$$
and we have 
$$
\partial_{n+2}q(g_1,\dots,g_{n+3})=|G| \cdot f(g_1,\dots,g_{n+3})
$$
and, for every $v\in G$,
\begin{multline*}
q(v\cdot u_1,\dots, v\cdot u_{n+2})=\sum_{g\in G} h_g(v\cdot u_1,\dots,v\cdot u_{n+2})=
\sum_{g\in G}h(gv\cdot u_1,\dots,gv\cdot u_{n+2})\\
=\sum_{w\in G}h(w\cdot u_1,\dots,w\cdot u_{n+2})=q(u_1,\dots,u_{n+2})
\end{multline*}
and thus
$q\in\CI(G^{n-1})$. We conclude that
$$
|G|\cdot[f]=0\ .
$$

In particular, $H^n(X\mid Y)$ is a torsion group. On the other hand, since $K$ is finite, $\CC(K^{n+3})$ is a finitely generated free abelian group, therefore its subgroup $\ker(d_{n+3})$ is also a finitely generated free abelian group, and
$H^n(X\mid Y)\cong \ker(d_{n+3})/\range(d_{n+2})$ is a finitely generated abelian group. As it is 
a torsion group, it is finite.

We consider now the case of a Cantor isometric extension. We have $K=G/H$ where $G$ is a Cantor group. 
First we check that $G$ is an inverse limit of finite groups.

We claim that any  neighborhood $U$ of the unit $e_G$ in $G$ contains an open normal subgroup $R$ of $G$. Since $G$ is a Cantor group, $U$ contains a clopen neighborhood $U'$ of  $e_G$.
 Since $U'$ is compact and open, there exists a symmetrical neighborhood $V$ of $e_G$ in $G$ with $U'V\subset U'$. If $Q$ is the group spanned by $V$, then $Q$ is an open subgroup of $G$, included in $U'$ and thus in $U$.
 By 
  compactness, $Q$ has a finite index in $G$. We chose a finite set $\{g_1,\dots,g_k\}$ in $G$ such that $G$ is the union of the  cosets $g_iQ$, and define $R$ to be the intersection of the sets $g_iQg_i^{-1}$. Then $R$ satisfies the announced properties.
 
Let $(U_j)$ be a decreasing sequence of  neighborhoods of $e_G$ in $G$, with $\cap_jU_j=\{e_G\}$.
For every $j$, we chose an open normal subgroup $R_j$ included in $U_j$, and we modify these subgroups so that the sequence $(R_j)$ is decreasing. Then all the groups $G/R_j$ are finite and $G$ is their inverse limit.

Now, the Cantor extension $X=Y\times G/H$ of $Y$ is the inverse limit of the finite extensions $X_j=Y\times(G/HR_j)$, defined by the reduction of the cocycle $\sigma$ modulo $R_j$. We have 
shown that for every $n$, $H^n(X\mid Y)$ is the direct limit of the groups 
$H^n(X_j\mid Y)$ and thus is a torsion group.
  \end{proof}
  
\subsection{The case of an extension by a finite group}
\label{subsec:finite-groups}
\label{subsec:ext_finite_group}
\begin{theorem}
\label{th:groupe}
Let $(Y,S)$ be a Cantor minimal system and let $(X,T)$ be a minimal extension of $(Y,S)$ by a finite group $G$. Then, for every $n\geq 0$,
$$
H^n(X\mid Y)\cong H^{n+2}(G)\ .
$$
\end{theorem}
\begin{proof}
Consider the so called ``bar resolution of $\Z$ in homogeneous form'':
$$
\cdots\rightarrow \overline{B}_n
\overset{\delta_n}{\rightarrow}
\overline{B}_{n-1}    \rightarrow\cdots\rightarrow 
\overline{B}_1
\overset{\delta_1}{\rightarrow}
\overline{B}_0
\overset{\epsilon}{\rightarrow}\Z\to 0
$$
where, for every $n\geq 0$, $\overline B_n$ is the free abelian group on the set of $(n+1)$-tuples of elements of $G$, endowed with the left diagonal action of $G$:
$$
h\cdot(g_1,g_2,\dots,g_{n+1})=(hg_1,hg_2,\dots,hg_{n+1})\ .
$$
$\delta_n\colon\overline B_n\to\overline B_{n-1}$ is the usual simplicial boundary map:
$$
\delta_n(y_1,y_2,\dots,y_{n+1})=\sum_{j=1}^{n+1}(-1)^{j+1}(y_1,y_2,\dots,\widehat{y_j},\dots,y_{n+1})\ .
$$
$\epsilon\colon \overline B_0=\Z G\to\Z$ is the augmentation homomorphism, given by $\epsilon(g)=1$ for every $g\in G$. 

Now, the cohomology $H^*(G)$ is the homology of the chain
\begin{multline*}
0\rightarrow \hom_G(\overline{B}_0,\Z) 
\overset{\delta_1^*}{\rightarrow}
\hom_G(\overline{B}_1,\Z) \rightarrow\cdots\\
\rightarrow\hom_G(\overline{B}_{n-1},\Z)   
\overset{\delta_n^*}{\rightarrow}
 \hom_G(\overline{B}_n,\Z)\rightarrow\cdots
\end{multline*}

Since $\Z$ is endowed with the trivial action of $G$, for every $n\geq 0$ we can identify
$\hom_G(\overline B_n,\Z)$ with $\CI(G^{n+1})$ and the chain above with the chain
\begin{equation}
\label{eq:CIG}
0\to\CI(G)\overset{d_1}{\rightarrow}  \CI(G^2) 
 \overset{d_2}{\rightarrow}  \CI(G^3)
    \overset{d_3}{\rightarrow}\cdots
\end{equation}
which is a particular case of the chain~\eqref{eq:CI}.
Theorem~\ref{th:groupe} follows now immediately from
Corollary~\ref{cor:isom}.
\end{proof}

\begin{comments}
Given any finite group $G$ we can, e.g. by using Theorem \ref{thm:Toeplitz},
construct a minimal extension $p \colon (X,T) \to (Y,S)$ by the group $G$.
In this way we embed, via Theorem \ref{th:groupe},
the entire ``algebraic'' cohomology of the finite group $G$ in the
``dynamical cohomology'' of the extension $p$:
$$
H^n(X\mid Y)\cong H^{n+2}(G)\ .
$$
The maximal intermediate extension $(\zab,\rab)$ of $(Y,S)$ by an abelian group (see Section~\ref{subsec:inter_ab}) is an extension by 
the abelianized group $\gab:=G/G_2$ of $G$, where $G_2$ is the commutator subgroup of $G$. By Theorem~\ref{th:abel-max}, $\widehat{\gab}$ is the torsion subgroup of $K^0(X)/p^*K_0(Y)$, which is equal to the torsion subgroup of $H^0(X\mid Y)$ by Corollary~\ref{cor:torsion}. Finally, by Theorem~\ref{th:groupe}, 
$$
\widehat{\gab}=H^0(X\mid Y)=H^2(G)\ .
$$
For completeness, we give a purely algebraic proof of the equality $\widehat{\gab}=H^2(G)$
in Appendix~\ref{app:C}.

In particular, we recover here the fact that
$$\text{
$H^2(G)=0$ when $G$ is a finite simple non abelian group}
$$
or more generally a finite group with $G_2=G$. 
In Theorem \ref{simple} we have shown the existence of 
an extension by a simple finite non abelian group $G$ such that 
$K_0(X,T)/p^*K_0(Y,S)$ is not trivial (this group is torsion free by Theorem~\ref{th:finite-abel}). A stronger result would have been the existence of an extension by a simple finite non abelian group $G$ such that $H^0(X\mid Y)$ is not trivial. 
This approach however fails since
$H^2(G)$  always vanishes for such groups.
\end{comments}

\subsection{Back to the cyclic case}
It is well known that 
$$
H^n(\Z_k)=\begin{cases}
\Z_k & \text{if $n$ is even}\\
0 & \text{if $n$ is odd}
\end{cases}
$$
Therefore, if $(X,T)$ is a minimal extension of $(Y,S)$ by $\Z_k$ we have
$$H^0(X\mid Y)=H^2(\Z_k)=\Z_k$$
and, by Corollary~\ref{cor:torsion},
$$\torsion\bigl(K_0(X,T)/p^*K_0(Y,S)\bigr)=\Z_k\ .$$
This puts another light onto theorem~\ref{th:abel-max} and  Proposition~\ref{prop:cyclic}.

\begin{comments}
It would be interesting to have more information about the
non-torsion part of 
$K_0(X,T)/p^*K_0(Y,S)$ in the case of an extension by a finite group or, more generally, of a  finite or Cantor isometric extension.
\end{comments}

\subsection{The absolute case}
\label{subsec:abs2}
As in Section~\ref{subsec:abs1} and with the same notation, we consider the case that $Y$ is the trivial system consisting in a single point. For every $n$, we write $H^n(X)$ instead of $H^n(X\mid Y)$.

Let $r\colon (X,T)\to(\zrat,\rrat)$ be the factor map on the rational equicontinuous factor of $X$. It is not difficult to check that, for every $n$,  every invariant integer valued function on $X^n$ factorizes through $\zrat^n$, that is, can be written as 
$f\circ r ^{\times n}$ for some invariant function $f$   on $(\zrat^n,\rrat^{\times n})$.
By Proposition~\ref{prop:chasing}, the natural map $H^n(\zrat)\to H^n(X)$ is an isomorphism. Here, $H^n(\zrat)$ is interpreted in the dynamical sense, but by Theorem~\ref{th:groupe}, it is equal to $H^{n+2}(\zrat)$
 interpreted in the group theoretical sense: $\zrat$ is the inverse limit of a sequence finite cyclic groups $(\Z/n_i\Z)$, and $H^{n+2}(\zrat)$ is the direct limit of their cohomology. We get 
 $$
 H^n(X)=\begin{cases}
 \widehat{\zrat} =\torsion\bigl(K_0(X)/\Z\bigr) & \text{if   $n$ is even}\\
 0 & \text{if $n$ is odd.}
 \end{cases}
 $$

\section{The groups $H^n(X \mid Y)$ are torsion groups}
\label{sec:torsion}
In this Section we show:
\begin{theorem}
\label{thm:H=torsion}
Let 
$p\colon (X,T)\to (Y,S)$ be a factor map between minimal Cantor systems.
Let $(X,T) \overset{r} \to (\zisoc,\risoc) \overset{q}\to (Y,S)$, with $p = q \circ r$, be the largest intermediate Cantor isometric extension between $X$ and $Y$. Then:
\begin{enumerate}
\item
\label{it:HnXYZY}
For every $n \ge 0$, $H^n(X \mid Y)=H^n(Z \mid Y)$.
\item 
\label{it:HnTorsion}
For every $n \ge 0$,  $H^n(X \mid Y)$ is a torsion group.
\item
\label{it:HzeroTorsion}
$\displaystyle
H^0(X\mid Y)=\torsion\Bigl(\frac{K_0(X)}
{p^*K_0(Y)}\Bigr)\ .
$
\end{enumerate}
In particular, taking $Y=\{\text{pt}\}$ to be the trivial one point system, we have
$$
H^0(X)=\torsion(K_0(X)/\Z)\ ,
$$
where we identify $K_0(\text{pt})$ with $\Z$.
\end{theorem}

We first recall the definition of the relative regionally proximal relation.

\begin{definition}\label{def:QE}
\begin{enumerate}
\item
Let $p\colon (X,T) \to (Y,S)$ be a factor map between minimal systems then the {\em relative regionally proximal relation $Q_p\subset X^2_p$} is  the set of pairs $(x,x')\in X_p^2$ such that there exist a sequence $(x_j,x'_j)\in X^2_p$ converging to $(x,x')$ and a sequence $(n_j)$ of integers such that 
$(T^{n_j}x_j,T^{n_j}x'_j)$ converges to $(x,x)$.
We remark that $Q_p$ is a closed invariant relation on $X$.
\item
The smallest invariant closed equivalence relation 
(icer) $R \subset X_p^2$
such that $X/R$ is an isometric extension of $Y$ is called {the relative
equicontinuous structure relation} and is denoted by $E_p$. Clearly
$Q_p \subset E_p$ and moreover $E_p$ is the least 
icer which contains $Q_p$. 
The latter assertion is well known (e.g. it follows directly 
from \cite[Corollary 10, Chapter 7]{A});
for completeness we supply a proof in Appendix \ref{app:E}.
\end{enumerate}
\end{definition}

We will need the following lemma concerning $Q_p$.

\begin{lemma}
\label{lem:Q}
Let $p\colon (X,T) \to (Y,S)$ be a factor map between Cantor minimal systems
Let $n \geq 0$ and $f \in \CI(X^n_p)$. If $(x_i,x_i') \in  Q_p$
then for every $(x_1,x_2,\dots,x_n) \in X^n_p$ we have
$f(x_1,x_2,\dots, x_i, \dots ,x_n)=
f(x_1,x_2,\dots,x'_i,\dots, x_n)$.
\end{lemma}

\begin{proof}
To keep the notation simple we assume $n=2$ and $i=1$.
The assumption $(x_1,x_1') \in  Q_p$ means that there exist a sequence 
$(x_{1,j},x'_{1,j})$ in $X^2_p$ converging to 
$(x_1,x_1')$
and a sequence $(n_j)$ of integers with
$(T^{n_j}x_{1,j},T^{n_j}x'_{1,j}) \to (w,w)$ for some point $w \in X$.

Now the set-valued map $p^{-1} \colon Y \to 2^X$, the latter being the 
compact metric space of closed 
subsets of $X$ equipped with the
Hausdorff distance, is an upper-semi-continuous map,
meaning that $\limsup 
p^{-1}\{y_k\} \subset p^{-1}\{y\}$,
whenever $y_k \to y$ in $Y$. 
Therefore it admits a dense $G_\delta$ subset
$Y_c \subset Y$ of continuity points (see e.g.~\cite{C}).

We first assume that the point $y_0=p(x_2)=p(x_1)$ is in $Y_c$.
Since $p(x_2)=p(x_1)$, $x_{1,j}\to x_1$ and $p(x_1)=y_0$ is a continuity point
for $p^{-1}$, 
we can choose a sequence $(x_{2,j})$ converging to $x_2 \in p^{-1}\{y_0\}$ 
such that $p(x_{2,j})= p(x_{1,j})$ and thus also $p(x_{2,j})=p(x'_{1,j})$ for every 
$j$, that is, $(x_{1,j}, x_{2,j})\in X^2_p$ and $(x'_{1,j}, x_{2,j})\in X^2_p$. 
We can further assume that the sequence $(T^{n_j}x_{2,j})$ converges to some point $z$. We now have
\begin{align*}
& X^2_p \ni (x_{1,j}, x_{2,j}) \to (x_1,x_2)\ ; \quad  
(T^{n_j}x_{1,j}, T^{n_j}x_{2,j}) \to (w,z)\ ;\\
& X^2_p \ni (x'_{1,j}, x_{2,j}) \to (x'_1,x_2)\ ; \quad  
(T^{n_j}x'_{1,j}, T^{n_j}x_{2,j}) \to (w,z)\ .
\end{align*}
Since $f$ is locally constant, there is an $j$ with
\begin{align*}
f(x_1, x_2) & =  f(x_{1,j}, x_{2,j}) = f(T^{n_j}x_{1,j}, T^{n_j}x_{2,j}) = f(w, z) \\
& = f(T^{n_j}x'_{1,j}, T^{n_j}x_{2,j}) = f(x'_{1,j}, x_{2,j})  = f(x'_1, x_2)\ .
\end{align*}

Finally if $p(x_1) = p(x_2) = y$ is not in $Y_c$, we pick a point $y_0 \in Y_c$
and a sequence $(m_j)$ such that $T^{m_j}y \to y_0$. 
Passing to a subsequence, we can assume that $T^{m_j}x_1\to\tilde x_1$, $ T^{m_j}x_2\to\tilde x_2$ and $T^{m_j}x'_1\to \tilde x_1'$ for some $\tilde x_1,\tilde x_2$ and $\tilde x'_1\in X$. Then
$$
(T^{m_j}x_1,T^{m_j}x_2) \to (\tilde{x}_1,\tilde{x}_2), \qquad
(T^{m_j}x'_1,T^{m_j}x_2) \to ({\tilde{x}_1}',\tilde{x}_2),
$$
$p(\tilde{x}_1)=p(\tilde{x}_2) = y_0$, and $(\tilde{x}_1,{\tilde{x}}'_1) \in Q_p$ since $Q_p$ is closed and invariant. As we have seen above this implies
that $f(\tilde{x}_1,\tilde{x}_2) = f({\tilde{x}_1}',\tilde{x}_2)$.
But then also 
$$
f(x_1,x_2) =f(\tilde{x}_1,\tilde{x}_2) = f({\tilde{x}_1}',\tilde{x}_2)
=f(x'_1,x_2)
$$ 
and the proof is complete.
\qed
\renewcommand{\qed}{}
\end{proof}

\begin{corollary}\label{cor:Q}
$p\colon (X,T) \to (Y,S)$  be a factor map between Cantor minimal systems.
Let $n \geq 0$ and $f \in \CI(X^n_p)$. Let
$(X,T) \overset{r} \to (\zisoc,\risoc) \overset{q}\to (Y,S)$ be the largest intermediate Cantor isometric extension between $X$ and $Y$. Then $f$ factors through $\zisoc$, that is,
there exists a function $\tilde{f} \in \CI((\zisoc)_q^n)$ with $f = \tilde{f}
\circ r^{\times n}$. 
\end{corollary}

\begin{proof}
Let $X \overset{r'} \to (\ziso,\riso) \overset{q'}\to Y$ be the largest intermediate isometric extension between $X$ and $Y$.
By Lemma \ref{lem:Q} and the fact that the factor $(\ziso,\riso)$
corresponds to the least closed invariant equivalence relation on $X \times X$
which contains $Q_p$ (Appendix~\ref{app:E}), we deduce that $f$ factorizes through 
$\ziso$.
Since $f$ is locally constant, $f$ factorizes through $\zisoc$.
\end{proof}

\begin{proof}[Proof of Theorem~\ref{thm:H=torsion}]
\ref{it:HnXYZY}
This follows from Corollary \ref{cor:Q} and Proposition \ref{prop:chasing}.

\ref{it:HnTorsion}
Apply Proposition \ref{th:tosion_Isometric} .

\ref{it:HzeroTorsion} By Corollary \ref{cor:torsion}, 
$$
\torsion\bigl(H^0(X \mid Y)\bigr)
=\torsion(K_0(X) / p ^*K_0(Y)),
$$
and the statement follows from~\ref{it:HnTorsion}.
\end{proof}

\begin{remark}
In fact, by Remark \ref{alln}, we conclude in the same way that for all $n$
$$
H^n(X\mid Y) = \torsion\Bigl(\frac{K_0(X^{n+1}_p,T^{\times (n+1)})}
{\partial_n^*( K_0(X^n_p,T^{\times n}))}\Bigr)\ .
$$
\end{remark}

\appendix
\section{An example: the Morse minimal system}
\label{ap:morse}

We recall here some well known facts about the Morse dynamical system, and compute the algebraic invariant introduced in the paper in this case. 
We refer to~\cite{Q} for the study of the Morse system and more generally of substitution dynamical systems, and to~\cite{DHS}  for the computation of 
the dimension groups of these systems.
\begin{fact}
The Morse minimal system $(X,T)$ admits the dyadic adding machine as its
maximal isometric factor $p\colon(X,T) \to (Y,S)$ and has the following structure
$$ 
(X,T)\overset {r}{\rightarrow} (Z,R)\overset{q}{\rightarrow} (Y,S)\ ,
$$
where $(Z,R)$ is a Toeplitz (and a substitution) system and $p = q \circ r$.
\end{fact}

The points of $X$ are written $x=(x_n\colon n\in\Z)$; Sometimes we consider $x$ as belonging to $\{0,1\}^\Z$, sometimes as belonging to $(\Z_2)^\Z$. 
The same convention holds for $Y$.

\noindent\textbf{a)}
$(X,T)$ is the 	system of the Morse substitution
$$
\sigma\colon \left\{\begin{matrix} 0 & \to & 01\\ 
                              1 & \to & 10
\end{matrix}\right.
$$
The dimension group of $(X,T)$ is computed in~\cite{DHS}. 
\begin{fact}
\label{fa:KX}
$K_0(X,T)$ is the dimension group of the matrix
\begin{gather*}
A=\begin{pmatrix}
0 & 2 \\
1 & 1
\end{pmatrix}
\text{ with unit }
e_X=\begin{pmatrix} 2 \\ 2 \end{pmatrix}\ .
\\
K_0(X,T)=\Bigl\{ \begin{pmatrix} (a+2b)/3 \\ (a-b)/3\end{pmatrix} 
\colon b\in\Z,\ \exists n \ 2^na\in\Z,\ 2^na=(-1)^nb\bmod 3\Bigr\}\ .
\end{gather*}
\end{fact}

\noindent\textbf{b)}
$(Z,R)$ is the system of the substitution:
$$
\zeta\colon \left\{\begin{matrix} 1 & \to & 10\\ 
                              0 & \to & 11
\end{matrix}\right.
$$

\begin{fact}
\label{fa:KY}
$K_0(Z,R)$ is the dimension group of the matrix
\begin{gather*}
B=\begin{pmatrix}
1 & 2 \\
1 & 0\end{pmatrix}
\text{ with unit }
e_Y=\begin{pmatrix} 2 \\ 1 \end{pmatrix}\ .
\\
K_0(Z,R)=\Bigl\{ \begin{pmatrix} (2a+b)/3 \\ (a-b)/3\end{pmatrix}\colon
b\in\Z,\ 
\exists n\ 2^na\in\Z,\ 2^na=(-1)^nb\bmod 3\Bigr\}\ .
\end{gather*}
\end{fact}

\noindent\textbf{c)} Let us consider $X$ and $Z$ as included in $(\Z_2)^\Z$. 
The factor map $r\colon X\to Z$ is the map defined by the code
$$ z_n=x_n+x_{n+1}\ .$$
We  identify $X$ with $Z\times \Z_2$ by the map $x\mapsto(r(x),x_0)$.
Then $r\colon (X,T)\to (Z,R)$ is an extension by $\Z_2$, given by the cocycle $\sigma(y)=z_0$.

\begin{fact}
\label{fa:HXZ}
$H^0(X\mid Z)=\Z_2$.
\end{fact}
In fact, we
present an explicit element of order $2$ in $K_0(X,T)/r^*K_0(Z,R)$.
Let $h(x)=x_0$, considered as an element of $\{0,1\}$, and $g(y)=y_0$,  considered as an element of 
$\{0,1\}$. We have 
$$h(Tx)-h(x)=g\circ r(x)-2\cdot 1_{[10]}(x)\ .$$
If $\alpha$ is the class of $1_{[10]}$ in $K_0(X,T)$ we have that $\alpha\notin r^*K_0(Z,R)$ and $2\alpha\in r^*K_0(Z,R)$.
In the representation of Fact~\ref{fa:KX}, $\alpha$ 
corresponds to the vector $\begin{pmatrix}
0 \\ 1\end{pmatrix}$.

\noindent\textbf{d)}
The map $r^*\colon K_0(Z,R)\to K_0(X,T)$ is not computed in 
\cite{DHS}
but it can be obtained 
using the methods of the present work. In the representations of Facts~\ref{fa:KX} and~\ref{fa:KY}, the map $r^*$ is 
given by the matrix
$$
R=\begin{pmatrix} 2 & -2\\ 0 & 2\end{pmatrix}\ .
$$
We can check that $Re_Y=e_X$ and that $RB=AR$. If $\alpha\in\R^2$ belongs to the dimension group of $(Z,R)$, that is, if it is represented by a vector $u\in\R^2$ with $B^n u\in\Z^2$ for some $n$, then 
$A^nRu=RB^nu\in 2\Z^2$ and this means that $r^*\alpha\in 2\cdot K_0(X,T)$. The converse inclusion is obtained in the same way and thus 
$$
r^*K_0(Z,R)=2\cdot K_0(X,T)\ .
$$
In the representation of Fact~\ref{fa:KX}, $r^*K_0(Z,R)$ corresponds to the case that $b\in 2\Z$. Therefore
\begin{fact}
$\displaystyle
\frac{K_0(X,T)}{r^*K_0(Z,R)}=\Z_2\ .
$
\end{fact}
This is coherent with Fact~\ref{fa:HXZ} and 
Part~\ref{it:HzeroTorsion} of Theorem\ref{thm:H=torsion}.

\noindent\textbf{e)} The dimension group of the odometer $(Y,S)$ is well known.
Here $\Z[1/2]$ is considered as a subgroup of $\Q$: 
$$\Z[1/2]=\{x\in\Q\colon \exists n\in\N,\ 2^nx\in\Z\}\ .$$
$$
K_0(Y,S)=\Z[1/2]\ ;\ e_Y=1\ ;\ K_0^+(Y,S)=\Z[1/2]\cap \Q^+\ .
$$
Using the representation of Fact~\ref{fa:KY} for $K_0(Z,R)$, the map $q^*\colon K_0(Y,S)\to K_0(Z,R)$ is
$
u\mapsto u\cdot\begin{pmatrix} 2 \\ 1\end{pmatrix}\ .
$

In the representation of Fact~\ref{fa:KY}, $q^*K_0(Y,S)$ is the subgroup of $K_0(Z,R)$ corresponding to $b=0$. The map $\begin{pmatrix} u \\ v\end{pmatrix}\mapsto u-2v$ induces an isomorphism
\begin{fact}
\label{fa:ZY}
$\displaystyle
\frac{K_0(Z,R)}{q^*K_0(Y,S)}\cong\Z\ .
$
\end{fact}
Since the factor map $q\colon (Z,R)\to (Y,S)$ is almost 
one-to-one, we have 
\begin{fact}
$H^*(Z\mid Y)=0$.
\end{fact}
In particular, $H^0(Z\mid Y)=0$, which is the expected result given Fact~\ref{fa:ZY} and
Part~\ref{it:HzeroTorsion} of Theorem\ref{thm:H=torsion}.

\noindent\textbf{f)} 
It is not hard to see that $\mathcal{I}(X^3_p) = \Z$. By Proposition \ref{prop:chasing}
this implies:
\begin{fact}
\label{fa:HXY}
$H^0(X \mid Y)=0$.
\end{fact}

Notice that 
$$
H^0(X \mid Y) \neq H^0(X \mid Z) / H^0(Z \mid Y) =
(\Z_2)/ 0 = \Z_2.
$$

\noindent\textbf{g)} From the preceding descriptions of $r^*$ and $q^*$ it follows that
$p^*K_0(Y,S)$ is the subgroup of $K_0(X,T)$ given by  Fact~\ref{fa:KX} corresponding to $b=0$.
Therefore we have
\begin{fact}
$K_0(X,T)/p^*K_0(Y,S)\cong \Z$.
\end{fact}
This is coherent with Fact~\ref{fa:HXY} and Part~\ref{it:HzeroTorsion} of Theorem\ref{thm:H=torsion}.

\section{Intermediate extensions by $\Z_n$}
\label{sec:appendix}
We present a particular case of Theorem~\ref{th:finite-abel} with a simpler proof.
\begin{proposition}
\label{prop:cyclic}
Let $p:(X,T)\to(Y,S)$ be factor map between  Cantor
minimal systems, and
$n>1$ an integer. The following statements are equivalent:
\begin{enumerate}
\item\label{it:cyclic1}
$p$  admits an intermediate extension by $\Z_n$ .
\item
\label{it:cyclic2}
$K_0(X,T)/p^*K_0(Y,S)$ contains an element of order $n$.
\end{enumerate}
\end{proposition}
\begin{proof} 
\ref{it:cyclic1} $\Rightarrow$ \ref{it:cyclic2}

Without loss of generality, we can assume that $p\colon (X,T)\to(Y,S)$ is itself
an extension by $\Z_n$, that is: $X=Y\times \Z_n$, $p$ is the first projection, 
and $T(y,j)=(Sy,j+\sigma(y)\bmod n)$ for some continuous map $\sigma\colon Y\to \Z_n$.

We identify $\Z_n$ with the subset $\{0,1,\ldots,n-1\}$ of $\Z$ and
consider the second projection 
$h\colon  X\to \Z_n$ as valued in $\Z$. We have that $h$ belongs to $\CC(X)$ and, for
every $x\in X$, $h(Tx)-h(x)=\sigma\circ p(x)\bmod n$. Thus there exists
$f\in\CC(X)$ with
\begin{equation}
\label{eq:n}
h\circ T-h=\sigma\circ p+nf\ .
\end{equation}
Let $[f]$ be the image of $f$ in $K^0(X,T)$. By~\eqref{eq:n}, we have $n[f]\in p^*K^0(Y,S)$.

Let $k$ be the smallest positive integer such that $k[f]\in p^*K^0(Y,S)$. We claim that $k=n$.
There
exists $u\in\CC(X)$ and $w\in\CC(Y)$ such that
\begin{equation}
\label{eq:n2}
u\circ T-u=w\circ p+kf\ .
\end{equation}
Moreover, $k$ divides $n$, and we write $n=kd$ for some integer $d\geq 1$.
From~\eqref{eq:n} and~\eqref{eq:n2} we get
\begin{equation}
\label{eq:n3}
(h-du) \circ T -(h-du) = (\sigma-dw)\circ p\ .
\end{equation}
Thus $(\sigma-dw)\circ p$ is a coboundary in $\CC(X)$
and, by Lemma~\ref{lem:onetoone}, $\sigma-dw$ is a coboundary in $\CC(Y)$.
This however is impossible (because $\sigma$
is ergodic) except if $d=1$, that is, $k=n$.  The claim is proved.

This shows that $[f]$ defines an element of order $n$ in the
quotient $K^0(X,T)/ \allowbreak p^*K^0(Y,S)$.

\medskip\noindent\ref{it:cyclic2} $\Rightarrow$ \ref{it:cyclic1}.
Let $\alpha$ be an element of $K^0(X,T)$ corresponding to an element of
order $n$ in $K^0(X,T)/p^*K^0(Y,S)$. So, $n\alpha\in K^0(Y,S)$ and
$d\alpha\notin K^0(Y,S)$ for $1\leq d<n$. We choose $f\in\CC(X)$ such that
$[f]=\alpha$. 
There exists
$h\in\CC(X)$ and $\sigma\in\CC(Y)$ satisfying~\eqref{eq:n}. Let
$\bar\sigma\colon Y \to \Z_n$ be the reduction of $\sigma$ modulo $n$.

We claim that this $\Z_n$-cocycle is ergodic. Let $d\geq 1$ be a divisor of
$n$, and assume that the $\Z_n$-cocycle $d\sigma$ is a $\Z_n$-coboundary.
There exists $u\in \CC(Y)$ such that 
$d\sigma-nu$ is a coboundary.
We write $n=kd$. As $K_0(Y,S)$ is torsion-free, $\sigma-ku$ is a
coboundary. Substituting in~\eqref{eq:n} we get that $nf-ku\circ p\in p^*K_0(Y,S)$
and, by definition of $f$, $d=n$. Our claim is proved.

Let $q:(Y\times \Z_n,S_{\bar\sigma})\to(Y,S)$ be the extension defined by
$\bar\sigma$. We define $r\colon X\to (Y\times \Z_n)$ by
$r(x)=\bigl( p(x),h(x)\bmod n\bigr)$ where $h$ is as in~\eqref{eq:n}. 
Then it is immediate that $r$ is continuous, that $q\circ r=p=\pi$ and 
that $r$ commutes with the transformations. Moreover, $r$ is onto because
$(Y\times \Z_n,S_{\bar\sigma})$ is minimal. It is a factor map, and we get the announced intermediate extension.
\end{proof}

\section{Constructing $G$-extensions}
\label{app:B}
   
In this appendix we describe one way to construct an
extension $\pi\colon (X,T) \to (Y,S)$ of Cantor minimal systems
with the following properties.
\begin{enumerate}
\item
$\pi$ is a group extension with a given finite
noncommutative simple group $K$ as the group of the
extension.
\item
On $X$ there are two distinct invariant probability measures
which project onto the same invariant measure on $Y$.
Thus the induced map $\pi_*$ on the simplex of invariant
probability measures is not an isomorphism; whence the canonical image
of $K_0(Y)$ in $K_0(X)$ is a proper subgroup.
\end{enumerate}

Given a finite group $G$ our first task is to exhibit an extension $\pi \colon (X,T) \to (Y,S)$ between Cantor minimal systems which is a $G$ extension. That is
$G$ is a finite group of homeomorphisms of $X$ each of which commutes 
with $T$, and such that $X/G \cong Y$.

We will describe two such constructions. The first is explicit and produces  a 
$G$-extension of a Toeplitz system.
The second utilizes horocycle minimal systems.

\subsection{The Toeplitz construction}

We construct a Toeplitz system $(Y,S)$ and a cocycle $\phi \colon Y \to G$
such that the resulting dynamical system $(X,T)$ defined on 
the product space $X = Y \times G$ by
$T(y,g) = (Sy,g\phi(x))$ is minimal.
This construction was described to us by B. Weiss and we
thank him for the permission to publish it here.
For more details on Toeplitz systems we refer to S. William's work
\cite{Wi}.

\begin{theorem}\label{thm:Toeplitz}
Let $G$ be a finite group. There exists a strictly ergodic (hence minimal)
Toeplitz system $(Y,S)$ and a cocycle $\phi : Y \to G$ such that the associated
skew product on $X= Y \times G$ defined by 
$$
T(y,g) = (Sy,g\phi(x))
$$
is minimal.
\end{theorem}

\begin{proof}
Let $G = \{u_0, u_1,\dots, u_{N-1}\}$ be an enumeration of the elements of $G$,
with $u_0 =e$, the identity element.
The space $Y$ will be the orbit closure of a point $\omega \in 
\Omega = G^\Z$
under the shift map $S : G^\Z \to G^\Z$, $(S\omega')(n) = \omega'(n+1)$,
$\omega'\in \Omega$.
The cocycle $\phi : Y \to G$ will be the projection onto the zero
coordinate: $\phi(\omega') = \omega'(0)$, for every $\omega' \in Y$.

The point $\omega_0$ is defined inductively as follows. At stage $0$
we set $\omega(2n) = u_0$ for all $n \in \Z$.
At stage $1$ we set $\omega(4n +1) = u_1$ for all $n \in \Z$.
Suppose $\omega$ is already defined at all the integers of the form
$2^jn + 2^{j-1} -1$, $n \in \Z$, $1 \le j \le k$.
We set $a = \omega(0)\omega(1)\cdots \omega(2^k -2)$ and
$b = \omega(0)\omega(1)\cdots \omega(2^{k -1} -2)$.
We now define $\omega(2^{k+1}n + 2^k -1) = g_k$
with $g_k$ the unique solution of the equation
$$
ag_kb = u_{k'},
$$
where $k'$ is the unique element of $\{0,1,2,\dots,N-1\}$ such that $k' \equiv k \pmod{N}$.

This completes the construction of the sequence $\omega$, which is clearly
a regular Toeplitz sequence. It then follows that the dynamical system
$(Y,S)$ is minimal Cantor and uniquely ergodic (i.e. strictly ergodic).
Now the numbers of the form $2^kn + 2^{k-1} - 1$ are ``recurrence
times" for $\omega$ and it follows from our construction that each $u_j \in G$ 
is an essential value of the cocycle $\phi$. By the general theory of topological
cocycles we conclude that the skew-product $(Y\underset{\phi}\times G, T)$
is minimal
(see Lema\'{n}czyk and Mentzen \cite{LM}).
\end{proof}

\subsection{The $\SL(2,\R)$ construction}

Let $\G=\SL(2,\R)$ and let $\Gamma$ be a cocompact
discrete subgroup. It is well known that there exists a 
homomorphism $\phi:\Gamma\to F$ where $F$
is a free group on two generators. Let $E$
be a normal subgroup of $F$ with $G=F/E$.
let $\Sigma=\phi^{-1}(E)$; then $\Sigma$
is a normal subgroup of $\Gamma$
with $\Gamma/\Sigma=G$.
Set $X_1=\G/\Gamma$, $Y_1=\G/\Sigma$ and let
$\pi_1\colon X_1\to Y_1$ be the natural projection.
We let  $S_1\colon Y_1\to Y_1$ and $T_1:X_1\to X_1$ be defined by
multiplication on the left by the matrix
$$
h_1= \begin{pmatrix} 1&1\\ 0&1 \end{pmatrix}.
$$
It is well known that these ``time one"
horocycle transformations act minimally and in a
uniquely ergodic way. Moreover it is clear that 
under these actions the map $\pi$ becomes a group
extension with the group $G$ as the group of the
extension. For more details on the horocyclic flow see e.g.
\cite{Gl03}.

\begin{lemma}\label{Cantor}
If $(Y,S)$ is an infinite minimal dynamical system then there is
a Cantor minimal dynamical system $(X,T)$ and an almost
one-to-one factor map $\pi \colon (X,T) \to (Y,S)$.  
\end{lemma}

\begin{proof}
Let $f \colon C \to Y$ be a continuous surjection from the Cantor set $C$
onto $Y$. By Zorn's lemma we can assume that this map is 
{\em irreducible}; i.e. if $D \subset C$ is closed and $f(D)=Y$ then $D=C$.
Let $T$ denote the shift on $C^\Z$ and define
$$
X = \{x \in C^\Z \colon f(x_{n+1}) = S f(x_n),\ \forall n \in \Z\}.
$$
Let $\pi_n : X \to Y,\ n \in \Z$ be the projection maps restricted to $X$.
It is now easy to check that $X$ is a closed shift invariant subset
of $C^\Z$. Denoting 
by $\pi : X \to Y$ the map $\pi(x) =f \circ \pi_0(x)= f(x_0)$ we conclude that
$\pi \colon (X,T) \to (Y,S)$ is a factor map.

Next we will show that $\pi$ is irreducible. In fact, suppose $A \subsetneq
X$ is closed and satisfies $\pi(A) =Y$. Let $V  = X \setminus A$.
By the irreducibility of $f$
we get that $\pi_0(A) = \{x_0 : x \in A\} = C$. Next observe that also
$\pi(T^nA) = Y$ for every $n \in \Z$, whence also $\pi_n(A) =C$
for every $n \in \Z$. 

Now $V$ is a nonempty open subset of $X$ and therefore contains a
basic open set of the form 
$$
X \cap (\cdots \times C \times C \times U_n \times
U_{n+1} \times \cdots \times U_{n+k} \times C \times C \times \cdots),
$$
where the sets $U_j$ are nonempty open subsets of $C$ and at least one of 
them is not all of $C$.
Suppose $U_m\not=C$. Then $\pi_m (A) \cap U_m =\emptyset$,
contradicting the above observation.

Thus $\pi$ is indeed an irreducible map and it is easy to deduce that
$(X,T)$ is a minimal system.
Now it is well known that a factor map $\pi \colon (X,T) \to (Y,S)$
of minimal systems is irreducible if and only if it is almost one-to-one.
Since $X$ is clearly a Cantor set our proof is complete.
\end{proof}

The proof of the next relative disjointness lemma is standard.

\begin{lemma}\label{pull}
Let $\pi\colon (X,T) \to (Y,S)$ and $\sigma \colon(Z,R) \to (Y,S)$ be factor maps between
minimal systems. If $\pi$ is a distal extension and $\sigma$ an almost
one-to-one extension then the relative product
$$
X \underset{Y} \times Z = \{(x,z) \in X \times Z :  \pi(x) = \sigma(z)\},
$$
is a minimal dynamical system and the projection
$\theta \colon X \underset{Y} \times Z  \to Z$ is a distal extension. If moreover
$\pi$ is a $G$-extension with finite group $G$ then so is $\theta$.
\end{lemma}

Next construct a commutative diagram 
\begin{equation*}
\xymatrix
{
(X,T) \ar[d]_{\pi}\ar[r]  & (X_1,T_1) \ar[d]^{\pi_1}  \\
(Y,S) \ar[r] & (Y_1,S_1) 
}
\end{equation*}
where $(X,T)$ and $(Y,S)$ are minimal
Cantor and $\pi$ is a $G$-extension.
This is possible by applying Lemma \ref{Cantor} to find an
almost-one-one Cantor extension $Y$ of $Y_1$
and then taking $X = X_1 \underset{Y_1} \times Y$ and 
applying Lemma \ref{pull}.

\subsection{The Furstenberg Weiss theorem}

Next recall the following result 
of Furstenberg and Weiss, \cite{FW1}.

\begin{theorem}
Let $(Y,S)$ be a non-periodic dynamical system and
let $(X,T){\overset\pi\to} (Y,S)$
be an extension of $(Y,S)$, where $(X,T)$
is recurrent topologically transitive. Then there exist 
a minimal system $(\overline{X},\overline{T})$, an almost 1-1 extension
$(\overline{X},\overline{T})\overset {\overline{\pi}}\rightarrow (Y,S)$, 
a Borel $T$-invariant subset $X_0\subset X$, and a Borel measurable map
$X_0 {\overset \theta \to } \overline{X}$ 
satisfying
\begin{enumerate}
\item $\theta  T=\overline {T} \theta$
\item $\overline {\pi} \theta= \pi$
\item $\theta$ is a Borel isomorphism of $X_0$ onto its image
       $\overline{X}_0$ in $\overline{X}$
\item $\mu (X_0)=1$ for every $T$-invariant
measure $\mu$ on $X$.
In fact, the set $X_0$ is of the form $\pi^{-1}(Y_0)$
for some Borel $S$-invariant subset $Y_0 \subset Y$ with the property that
$\nu(Y_0)=1$ for any invariant measure $\nu$ on $Y$.
\end{enumerate}
\begin{equation*}
\xymatrix
{
(X,T) \ar[rr]^\theta \ar[dr]_{\pi} & & 
(\overline{X},\overline{T}) \ar[dl]^{\overline{\pi}} \\
& (Y,S) &
}
\end{equation*}
\end{theorem}

Note that when the system $(X,T)$ is Cantor minimal
so is $(\overline{X},\overline{T})$. Indeed, the fact that 
$\overline{X}$ is Cantor follows from the construction
in \cite{FW1}.

Now we consider the $G$-extension of minimal Cantor systems
$\pi \colon (X,T) \to (Y,S)$ and 
apply this theorem to the extension $\pi$ to obtain
a commutative diagram as in the theorem
with $(\overline{X},\overline{T})$ minimal, $\overline{\pi}$ an almost 1-1 extension, and $\theta$, defined on a full-measure (with respect
to every invariant probability measure) Borel subset $X_0 \subset
X$, a Borel isomorphism of $X_0$ onto its image $X'_0\subset \overline{X}$.

Our final step is to construct the relative
product system 
$$
(\tilde{X},\tilde{T})=(X \underset{ Y } \times \overline{X},T \times \overline{T})
$$
where 
$$
X \underset {Y} \times \overline{X} = \{(x,x')\in X\times \overline{X}:
\pi(x) =\overline{\pi}(x')\}.
$$
As the extension $\pi$ is isometric and the extension
$\tilde{\pi}$ is almost 1-1, it follows that $(\tilde{X},\tilde{T})$ is a minimal 
system (Lemma \ref{pull} again). We set $(\tilde{Y},\tilde{S})=
(\overline{X},\overline{T})$ and let 
$\tilde{\pi}\colon (\tilde{X},\tilde{T})\to (\tilde{Y},\tilde{S})$ be the natural projection of 
$\tilde{X}$ onto $\overline{X} = \tilde{Y}$.
Clearly now $\tilde{\pi}$ is a $G$-extension between the minimal Cantor systems
$(\tilde{X},\tilde{T})$ and $(\tilde{Y},\tilde{S})$.

We now use the measure isomorphism $\theta$ to construct the required measures.
Let $\mu$ be a fixed invariant measure on $X$ and let 
$\overline{\mu} =\theta_*(\mu)$. Then $\overline{\mu}$
is an invariant measure on $\overline{X}$ and both $\mu$ and 
$\overline{\mu}$ project onto the same invariant measure, say $\nu$, on $Y$.
Let
$$
\mu = \int_{Y} \mu_{y} \, d\nu(y), \qquad
\overline{\mu} = \int_{Y} \overline{\mu}_{y} \, d\nu(y),
$$
be the corresponding disintegrations.
The relative product measure
$$
\lambda = \int_{Y} \mu_{y} \times \overline{\mu}_{y} \, d\nu(y)
$$
is then an invariant measure on $\tilde{X}$ which projects under 
$\tilde{\pi}$ onto $\overline{\mu}$.
On the other hand the measure
$$
\eta = \int_{X_0} \delta_x \times \delta_{\theta(x)} \, d\mu(x),
$$
is a different invariant probability measure on $X$ projecting onto 
$\overline{\mu}$. 
In fact, disintegrating $\eta$ over $\overline{\mu}$ we see that
almost every fiber measure is a point mass, while in the disintegration
of $\lambda$ over $\overline{\mu}$ almost every fiber measure is an atomic measure with $|G|$ atoms. Thus by the uniqueness of disintegration, we conclude
that $\lambda$ and $\eta$ are indeed distinct measures.
It is now clear that the extension 
$\tilde{\pi}\colon(\tilde{X},\tilde{T})\to (\tilde{Y},\tilde{S})$
satisfies our requirements.

\section{On the relative regionally proximal relation}\label{app:E}

Let $\pi : X \to Y$ be a factor map between minimal systems.
We recall that the relative regionally proximal relation $Q_\pi$
and the relative equicontinuous structure relation $E_\pi$,
were defined in Section \ref{sec:torsion}, Definition \ref{def:QE}.

\begin{lemma}\label{lem:QQ}
Let $X \overset{\rho} \to Z \overset{\sigma} \to Y$ be factor maps between minimal systems with $\pi = \sigma \circ \rho$. Then 
$$
(\rho \times \rho)(Q_\pi) = Q_\sigma.
$$ 
\end{lemma}

\begin{proof}
Clearly $(\rho \times \rho)(Q_p) \subset Q_\sigma$. 
Suppose $(z,z') \in Q_\sigma$.
By the definition of $Q_\sigma$ there is a sequence 
$Z^2_\sigma \ni (z_i, z'_i) \to (w,w)$,
where $(w,w)$ is some point on the diagonal $\Delta_Z$, 
and a sequence 
$(n_i)$ with 
$(T^{n_i}z_i, T^{n_i}z'_i)$ tending to the point $(z,z')$.
Now the set-map $\rho^{-1} : Z \to 2^X$, the latter being the 
compact metric space of closed 
subsets of $X$ equipped with the
Hausdorff distance, is an upper-semi-continuous map.
Therefore 
it admits a dense $G_\delta$ subset
$Z_c \subset Z$ of continuity points. By the minimality of $Z$
we can assume that the point $w$ above is in $Z_c$.
Let $x_i \in X$ be such that $\rho(x_i) = z_i$ and we can assume, with no loss
in generality, that $x_i \to x$, for some point $x \in X$ with $\rho(x) = w$.

Since $w \in Z_c$ we have $\rho^{-1}(z'_i) \to \rho^{-1}(w)$
and therefore we can find a sequence $x'_i \in \rho^{-1}(z'_i)$ with
$x'_i \to x$. Thus $X^2_\pi \ni (x_i,x'_i) \to (x,x)$ and, again with no loss
in generality, we can assume that $(T^{n_i}x_i,T^{n_i}x'_i) \to (x,x')$
for some pair $(x,x') \in X^2_\pi$. But now $(x,x') \in Q_\pi$
and $(\rho \times \rho)(x,x') = (z,z')$.
Thus we conclude that indeed $(\rho \times \rho)(Q_\pi) = Q_\sigma$.
\end{proof}

\begin{theorem}
Let  $\pi : X \to Y$ be a factor map of minimal systems. Then 
$E_\pi$ is the least 
invariant closed equivalence relation
(icer) which contains the relation $Q_\pi$.
\end{theorem}

\begin{proof}
Let $E'_p$ denote the least 
icer which contains the relation $Q_\pi$.
Clearly then $E_\pi \supset {E'}_\pi$. 
Let $X \overset{\rho} \to Z' = X/{E'}_\pi \overset{\sigma}\to Y$ denote the natural
maps. By Lemma \ref{lem:QQ} we have $(\rho \times \rho)(Q_\pi) = Q_\sigma$.
But clearly $(\rho \times \rho)(Q_\pi) = \Delta_Z$ and it follows that
$Q_\sigma = \Delta_Z$. It is easy to check that the latter equality implies
that the extension $\sigma$ is an isometric (or equicontinuous) extension and we conclude that also $E_\pi \subset {E'}_\pi$.
\end{proof}

\section{A proof of Proposition~\ref{prop:algebra}}
\label{ap:algebra}

\noindent{\bf a)}
By 
the definition of the functor $\tor$,  
for every exact sequence
\begin{equation*}
\tag{\ref{eq:exactthree}}
0\to K\overset{j}{\rightarrow} L \overset{q}{\rightarrow}M\to 0
\end{equation*}
of abelian groups and  for every abelian group $G$, there exists an exact
sequence
$$
0\to \tor(M,G)\to K\otimes G\overset{j\otimes\id_G}{\rightarrow} 
L\otimes G
\overset{q\otimes\id_G}{\rightarrow} M\otimes G\to 0
$$
and this sequence depends on the sequence~\eqref{eq:exactthree} and on $G$ in a functorial
(covariant) way. Thus, in order to prove the first statement of the Proposition, we have only to construct, for every finite abelian
group $G$, an isomorphism
 $\Phi_G\colon\tor(M,G)\to\hom(\widehat G,M)$
depending on $G$ and $M$ in a functorial (covariant) way.

\medskip 
\noindent{\bf b)} 
 Let $G$ be a finite abelian group. In this part  we show that $\ext(G,\Z)\cong \widehat G$.  Let
\begin{equation}
\label{eq:resolution}
0\to A\overset{i}{\rightarrow}B\overset{p}{\rightarrow} G\to 0
\end{equation}
be a resolution of $G$, where $A$
and $B$ are finitely generated free abelian groups.

Let $i^*:\hom(B,\Z)\to\hom(A,\Z)$ be given by $i^*(f)=f\circ i$.
As $G$ is finite, $\hom(G,\Z)=0$ and $i^*$ is 
one-to-one. 
On the other hand, for every $f\in\hom(A,\Z)$ there exists a unique $\bar
f\in\hom(B,\Q)$ with $\bar f\circ i=f$.
Remark that $\bar f\in\hom(B,\Z)$ if and only if $f$ belongs to the range of
$i^*$. 
The map $f\mapsto\bar f$ is
a one-to-one homomorphism from $\hom(A,\Z)$ to $\hom(B,\Q)$.
Composing this homomorphism with the natural projection $\Q\to\Q/\Z$,
we get an homomorphism $\lambda:\hom(A,\Z)\to\hom(B,\Q/\Z)$. 

By construction, the kernel of $\lambda$ is the range of $i^*$. 
Moreover, the range of $\lambda$ is the  kernel of the  homomorphism 
$f\mapsto f\circ i$ from 
$\hom(B,\Q/\Z)$ to $\hom(A,\Q/\Z)$, thus this range can be identified with
$\hom(G,\Q/\Z)$. But, as $G$ is finite, every homomorphism $G\to\R/\Z$ 
takes its values in
$\Q/\Z$, and we have $\widehat G=\hom(G,\R/\Z)=\hom(G,\Q/\Z)$.

We get an exact sequence
\begin{equation}
\label{eq:three2}
0\to \hom(B,\Z)\overset{i^*}{\rightarrow}\hom(A,\Z)\to\widehat G\to 0
\end{equation}
depending on the exact sequence~\eqref{eq:resolution} and on $G$ in a functorial (contravariant) way. 
We conclude that indeed $\ext(G,\Z)=\widehat G$

\medskip
\noindent{\bf c)} In the next two parts we show~\eqref{eq:isomTor}, that is,  $\tor(M,G)\cong\hom(\ext(G,\Z),M)$. 
 Applying the functor 
  $\hom(\,\cdot\, ,M)$ to the sequence~\eqref{eq:three2}, we obtain an
exact sequence
$$
0\to \hom(\widehat G,M)\to \hom\bigl(\hom(A,\Z),M\bigr)\overset{\epsilon}{\rightarrow} 
\hom\bigl(\hom(B,\Z),M\bigr)
$$
depending on the exact sequence~\eqref{eq:resolution}, on $G$ and on $M$ in a functorial (covariant)
way. But $\hom(A,\Z)$ is a finitely generated free abelian group and
$\hom\bigl(\hom(A,\Z),\Z\bigr)$ can be identified with $A$. Moreover, 
 $\hom\bigl(\hom(A,\Z),M\bigr)$ can be identified with
$\hom\bigl(\hom(A,\Z),\Z\bigr)\otimes M$, thus with $A\otimes M$. In the
same way, we can identify  $\hom\bigl(\hom(B,\Z),M\bigr)$ with
$B\otimes M$. The map corresponding to the homomorphism $\epsilon$
above is simply 
$i\otimes\id_M$, and we get an exact sequence
\begin{equation}
\label{eq:tensor}
0\to \hom(\widehat G,M)\to A\otimes M\overset{i\otimes\id_M}{\rightarrow}
B\otimes M
\end{equation}
depending on the exact sequence~\eqref{eq:resolution}, on $G$ and on $M$ in a functorial (covariant)
way.

\medskip\noindent{\bf d)}
By the definition of the functor $\tor$, the exact sequence~\eqref{eq:tensor} induces
an  isomorphism 
\begin{equation}
\label{eq:isomTor}
\tor(M,G)\overset{\approx}{\rightarrow}\hom(\widehat G,M)
\end{equation}
depending on the sequence~\eqref{eq:resolution}, on $G$ and on $M$ in a functorial covariant way. But
any two presentations of $G$ of the form ~\eqref{eq:resolution} are chain equivalent. Looking at the constructions above,
it  follows easily that the isomorphism~\eqref{eq:isomTor} does not depend on the
resolution~\eqref{eq:resolution}, but only on $G$ and $M$. The first statement of the Proposition is proved.

\medskip
\noindent{\bf e)}
 Let  $G'$ be a proper subgroup of $G$.    As $K$ and
$L$ are  torsion free, $K\otimes G'$ can be considered as a subgroup of
$K\otimes G$, and $L\otimes G'$ as a subgroup of $L\otimes G$,
 and  we have 
$$
\ker(j\otimes\id_{G'})\;=\;\ker(j\otimes\id_G)\cap (K\otimes G')\ .
$$
If $\alpha$ is in $\ker(j\otimes\id_G)$,
then $\Phi_G(\alpha):\widehat G\to L/j(K)$ factorizes through the quotient 
$\widehat{G'}$, thus is equal to $0$ on the non trivial  subgroup $\widehat{G/G'}$ of $\widehat G$, and is not
one-to-one.

Conversely, let $\alpha\in\ker(j\otimes\id_G)$ be such that
$\Phi_G(\alpha)$ is not 
one-to-one. The kernel $H$ of this homomorphism is
a non trivial subgroup of $\widehat G$, thus it can be viewed as 
$\widehat{G/G'}$ for
some proper subgroup $G'$ of $G$, and $\widehat G/H=\widehat{G'}$.
$\Phi_G(\alpha)$ factorizes through $\widehat{G'}$, giving rise to some
$g\in \hom(\widehat{G'},L/j(K))$. Then $g=\Phi_{G'}(\beta)$ for some
$\beta\in\ker(j\otimes\id_{G'})\subset K\otimes G'\subset K\otimes G$,
and, using the functorial property of $\Phi$, it is easy to check that
$\beta=\alpha$. 

\medskip

For any automorphism $\phi$ of $G$,
 for any $\alpha\in \ker(j\otimes\id_G)$ we have:

$$
\Phi_G(\alpha)\circ\widehat\phi=\phi_*\Bigl(\Phi_G(\alpha)\Bigr)=
\Phi_G\Bigl((\id_K\otimes\phi)(\alpha)\Bigr)$$

Let $\alpha,\beta$ in $\ker(j\otimes\id_G)$ be such that 
$\Phi_G(\alpha)$ and $\Phi_G(\beta)$ are both 
one-to-one. These
homomorphisms have the same range if and only if there exists an
automorphism  $\psi$ of $\widehat G$ with $\Phi_G(\beta)=\Phi_G(\alpha)\circ
\psi$; the automorphisms of $\widehat G$ correspond by duality to the
automorphisms of $G$. Thus this condition is equivalent to the existence of
an automorphism $\phi$ of $G$ with 
$\Phi_G(\beta)=\Phi_G\Bigl((\id_K\otimes \phi)(\alpha)\Bigr)$, that is,
$\beta=(\id_K\otimes\phi)(\alpha)$.
\qed

\section{Vanishing $H^2(G)$}\label{app:C}

We provide a proof for the following theorem, following a method   indicated
to us by Benjy Weiss.

\begin{theorem}
Let $G$ be a finite group. Then $H^2(G)$ is the dual group $\widehat{\gab}$ of the abelianized group $\gab:=G/G_2$ of $G$, where $G_2$ is the commutator subgroup of $G$. In particular,
$H^2(G)=0$ for every finite noncommutative simple group $G$.
\end{theorem}

\begin{proof}
We use the notation and results 
of Section~\ref{subsec:ext_finite_group}, and in particular the chain~\eqref{eq:CIG}. 

First we define a
one-to-one homomorphism from $H^2(G)$ into 
$\widehat{\gab}$.
Let
$\phi\in\CI(G^3)$. We define $f\in\CC(G^2)$ by
$$
f(x,y)=\phi(1,x,xy)\ .
$$
Then $\phi$ 
belongs to $\ker(d_3)$ if and only of $f$ satisfies
\begin{equation}
\label{eq:ker}
f(y,z)-f(xy,z)+f(x,yz)-f(x,y)=0\text{ for all }x,y,z\in G
\end{equation}
and $\phi$
belongs to $\range(d_2)$ if and only if there exists $h\in\CC(G)$ with
\begin{equation}
\label{eq:range}
f(x,y)=h(x)+h(y)-h(xy)\ .
\end{equation}
Assume that $f$ satisfies~\eqref{eq:ker}. We write $k=|G|$. Defining $F\in\CC(G)$ by
$$
F(x) = \sum_{y \in G} f(x,y)\text{ for all }x,y\in G\ .
$$
and summing over $z$ in~\eqref{eq:ker} we get
\begin{equation}
\label{eq:kf}
k.f(x,y)=F(x)+F(y)-F(xy)\text{ for all }x,y\in G\ .
\end{equation}
Let $\tilde F$ be the composition of $F$ with the projection $\Z \to \Z_k$.
 We have
\begin{equation}
\label{eq:tildef}
\tilde F(xy)=\tilde F(x)+\tilde F(y)\text{ for all }x,y\in G\ .
\end{equation}
Let us consider $\Z_k$ as embedded in the torus $\T=\R/\Z$ in the natural way. Then the last formula means that $\tilde F$ is a group homomorphism from $G$ to $\T$. This homomorphism is trivial on the commutator subgroup $G_2$ of $G$ and this induces an homomorphism $j(f)\colon \gab\to\T$ that is, an element $j(f)$ of $\widehat{\gab}$.

The homomorphism $j(f)$ is trivial if and only if $\tilde F$ is identically zero, that is, if there exists a function $h\in\CC(G)$ with $F=k.h$. Putting this into~\eqref{eq:kf}, we get~\eqref{eq:range}.
It follows that the map $j$ induces
a one-to-one homomorphism, written $j$ also, from
$H^2(G):=\ker(d_3)/\range(d_2)$ into $\widehat{\gab}$.

It remains to prove that $j$ is onto. Let $\chi\in\widehat{\gab}$. $\chi$ 
be a homomorphism from $\gab$ to $\T$ and we lift it as  
a homomorphism, written $\chi$ also, from $G$ to $\T$. Since $G$ is finite, the range of $\chi$ is included in some finite subgroup of $\T$ and thus in $\Z_k$ for some $k\geq 1$. Therefore, $\chi$ is the composition of a homomorphism $\tilde F\colon G \to \Z_k$ with the natural inclusion of $\Z_k$ in $\T$, and $\tilde F$ satisfies~\eqref{eq:tildef}. 

Let $F\colon G\to\Z$ be a lift of $\tilde F$. Then there exists $f\in\CC(G^2)$ such that~\eqref{eq:kf} holds and it is immediate to check that~\eqref{eq:ker} holds. We define $\phi\in\CC(G^3)$ by
$$
\phi(u,v,w)=f(u^{-1}v,v^{-1}w)\ .
$$
Then $\phi$ belongs to $\CI(G^3)$, and in fact to $\ker(d_3)$ because~\eqref{eq:ker} holds. Following the preceding construction, we check that $j(f)=\chi$. This shows that the homomorphism $j$ is onto.
\end{proof}

\end{document}